\documentclass[11pt]{amsart}

\usepackage{amsmath,amsfonts,amssymb,amscd,amsthm,amsbsy}
\usepackage[alphabetic]{amsrefs}
\usepackage{amsaddr}

\usepackage{hyperref}
\usepackage{graphicx}
\usepackage{tikz-cd}
\numberwithin{equation}{section}

\newtheorem{thm}{Theorem}[section]

\newtheorem{lemma}[thm]{Lemma}
\newtheorem{cor}[thm]{Corollary}

\newtheorem{prop}[thm]{Proposition}
\newtheorem{con}[thm]{Conjecture}

\newtheorem*{Thm}{Theorem}
\newtheorem*{Cor}{Corollary}

\theoremstyle{remark}
\newtheorem{remark}[thm]{Remark}

\theoremstyle{definition}
\newtheorem{Def}[thm]{Definition}

\newcommand{\HH}{\mathcal{H}}

\newcommand{\FF}{\mathbb{F}}

\newcommand{\ZZ}{\mathbb{Z}}
\newcommand{\NN}{\mathbb{N}}

\newcommand{\p}{\mathfrak{p}}

\newcommand{\m}{\mathfrak{m}}

\DeclareMathOperator{\Hom}{Hom}

\DeclareMathOperator{\rank}{rank}

\DeclareMathOperator{\ind}{ind}

\DeclareMathOperator{\stab}{stab}

\DeclareMathOperator{\reg}{reg}

\DeclareMathOperator{\im}{image}
\DeclareMathOperator{\coker}{coker}
\DeclareMathOperator{\ass}{Ass}
\DeclareMathOperator{\ann}{ann}
\DeclareMathOperator{\Fix}{Fix}

\DeclareMathOperator{\depth}{depth}

\renewcommand{\to}[1][]{\xrightarrow{#1}}
\newcommand{\from}[1][]{\xleftarrow{#1}}
\newcommand{\gr}[1][]{\mathrm{gr}_{#1}}
\newcommand{\iso}{\xrightarrow{\sim}}

\newcommand{\colim}{\varinjlim}
\renewcommand{\lim}{\varprojlim}
\newcommand{\hookto}{\hookrightarrow}

\title[On the Duflot filtration]{On the Duflot filtration for  equivariant cohomology rings and applications to group cohomology}
\author{James C. Cameron}
\email{jcameron@math.ucla.edu}
\begin{document}
\begin{abstract}We study the Duflot filtration on the Borel equivariant cohomology of smooth manifolds with a smooth $p$-torus action. We axiomatize the filtration and prove analog of several structural results about equivariant cohomology rings in this setting. We apply this abstract theory to study the $\FF_p$ cohomology rings of classifying spaces of compact Lie groups, and show how to recover geometric results about the cohomology of $BG$ using equivariant cohomology. This includes some results about detection on subgroups and restrictions on associated primes that were previously only known for finite groups.

We are particularly interested in the local cohomology modules of equivariant cohomology rings, and we construct a tractable chain complex computing local cohomology. As an application, we study the local cohomology of the group cohomology of the $p$-Sylow subgroups of $S_{p^n}$ and give vanishing and nonvanishing results for these local cohomology modules that are sharper than those given by the current theory.
\end{abstract}
\maketitle

\tableofcontents

\section{Introduction}
% The geometry of group cohomology rings is of interest for both theoretical and computational reasons. 

Ever since the work of Quillen \cite{quillen}, the Borel equivariant cohomology of spaces with a $G$-action has been used to study group cohomology rings. By choosing the $G$-space appropriately, knowledge about equivariant cohomology can descend to group cohomology.

In fact, in order to study group cohomology rings it is enough to study the equivariant cohomology of smooth manifolds with a smooth $p$-torus action (by $p$-torus we mean a group of the form $(\ZZ/p)^n$). Duflot \cite{duflotfiltration} defined a filtration on the equivariant cohomology of a smooth $p$-toral manifold with the key property that the $i^{th}$ subquotient is free over the cohomology of a $p$-torus of rank $i$. 

This filtration was used by Duflot \cite{duflotassociatedprimes} to show that all associated primes of group cohomology rings come from restricting to $p$-tori, by Henn, Lannes, and Schwartz \cite{hlsunstablemodulesandmodpcohomology} to study group cohomology rings as unstable algebras over the Steenrod algebra, and by Symonds \cite{symondsregularity} to prove the regularity conjecture for group cohomology rings, which is a statement about the local cohomolgy modules of group cohomology rings. In each case, the result in question is reduced to the case of $p$-tori, where it can be established with direct computation.

The Duflot filtration puts a lot of extra structure on the $S$-equivariant cohomology of a smooth manifold, and the first part of this paper is devoted to studying this extra structure in an abstract setting. We give an axiomitization of the Duflot filtration which we call a \emph{free rank filtration}, and study depth, dimension, and associated primes in this abstract setting. Much of the geometric structure of a graded $\FF_p$-algebra is determined by its local cohomology modules, and we construct a tractable chain complex we call the \emph{Duflot chain complex} that computes local cohomology in the presence of a free rank filtration. Geometric information about the ring such as Krull dimension and restrictions on depth can be read directly from the Duflot chain complex.

We also axiomatize a refinement of the Duflot filtration to a filtration by a poset with a ranking function which refines the filtration by the natural numbers. The  Duflot filtration in equivariant cohomology has such a refinement, and we use this refinement to give a new construction of the Duflot filtration and to give results about restrictions on associated primes and detection on submanifolds. 

This refinement also lets us construct a spectral sequence that allows some computations of the local cohomology modules of the cohomology of iterated wreath products of $\ZZ/p$ with itself, i.e. the $p$-Sylow subgroups of $S_{p^n}$.
\subsection{Results}
All cohomology is with $\FF_p$ coefficients.
% \subsubsection{The Duflot chain complex}

Inspired by Symonds' proof of the regularity theorem, we show that the Duflot filtration of $H_S^*M$ for a smooth $S$-manifold $M$, with $S$ a $p$-torus, gives rise to a cochain complex, $DM$, of graded modules we call the Duflot complex of $M$. The cohomology of $DM$ is the local cohomology of $H_S^* M$, and $DM$ enjoys the following properties, which we state here in the special case that $G$ is a finite group and $M$ is the manifold $G \backslash U(V)$ for $V$ a faithful representation of $G$, with $S<U(V)$ the maximal $p$-torus of diagonal matrices of order $p$ acting on the right.

We denote local cohomology with respect to the (graded) maximal ideal by $\HH^*$.

\begin{Thm}[\ref{duflotcomplex},  \ref{depthanddimensionforHSM},  \ref{duflotandquilleningroupcohomology}]
A finite dimensional representation $G \hookto U(V)$ can be chosen so that the Duflot complex of $G \backslash U(V)$ satisfies the following:

\begin{enumerate}
\item $H^* D(G \backslash U(V)) = \HH^* H_S^*(G \backslash U(V))= \HH^*(H_G^*) \otimes H^* U(V)/S$.
\item $D(G \backslash U(V))^i=0$ for $i$ less than the $p$-rank of $Z(G)$ and $i$ greater than the $p$-rank of $G$.
\item The  Krull dimension of the $i^{th}$ term of the dual chain complex is $i$.
\item $D(G \backslash U(V))^{i,j}$ is zero above the line $i=-j+ \dim G \backslash U(V)$.
\end{enumerate}
\end{Thm}

From this theorem and basic commutative algebra we can conclude Quillen's theorem on the dimension of $H_G^*$, Duflot's lower bound for the depth of $H_G^*$, and Symonds' regularity theorem.
 % This method of proof is  different than the proofs of the theorems of Quillen and Duflot, but similar to Symonds' proof of the regularity theorem. %Also, even though the Duflot complex for $G \backslash U$ computes the local cohomology of $H_S^* G \backslash U$, not the local cohomology of $H_G^*$, 

% We also show that the layers of the Duflot filtration are good approximations to the local cohomology of $H_S^*M$ in the following sense.
% \begin{thm}[\ref{whytruncationsaregood}]
% Denote by $F_i$ the $i^{th}$ layer of the Duflot filtration for a smooth $S$-manifold $M$, where $S$ is a $p$-torus. Then $F_i H_S^*M \to H_S^*M$ induces a surjection $\HH^i F_i H_S^*M \to \HH^i H_S^*M$, and an isomorphism  $\HH^j F_i H_S^*M \to \HH^j H_S^*M$ for $j>i$. 
% \end{thm}
% This lets us compute some of the local cohomology modules of $H_S^*M$ by studying higher layers of the Duflot filtration, which can be easier in a way that we will shortly describe.

% \subsubsection{Systematizing the Duflot filtration}
The Duflot filtration puts a lot of structure on $H_S^*M$, and to clarify what structure is forced by the Duflot filtration we found it useful to systematize the Duflot filtration. We define  a ``free rank filtration'', which is simply a filtration with subquotients of the form appearing in the Duflot filtration, and we show that any module with a free rank filtration has a Duflot chain complex and satisfies analogs of the dimension, depth, and regularity theorems, as well as Duflot's theorem on associated primes. The  current application for this theory is to simplify the study of $S$-equivariant cohomology, but it would be very interesting to have more examples of algebras with free rank filtrations.

The Duflot filtration is a filtration by the natural numbers, but we show how it can be refined to be a filtration by a certain poset. This poset is the poset of connected components of $M^A$, as $A$ ranges over all subtori of $S$, and it is equipped with a map to $\NN^{op}$ by the rank of $A$. Precise definitions are given in sections \ref{stratifiedfiltrations} and \ref{propertiesofHSM}. 

Studying this refinement of the Duflot filtration yields some restrictions on associated primes and a detection result for cohomology.
It also lets us study some maps between $S$-manifolds in terms of the maps on their associated posets, which ultimately leads to new local cohomology computations.

Our detection result is the following:
\begin{Thm}[\ref{detectiononcentralizers}]
Let $G$ be a compact Lie group, and let $d$ be the depth of $H_G^*$. 
Then $H_G^* \to \prod_{E<G,\rank E=d} H^*_{C_G E}$ is injective.
\end{Thm}
This result is due to Carlson \cite{carlsondepthconjecture} for finite groups. 
% This result could also be derived using the techniques of Henn, Lannes, and Schwartz in \cite{hlsunstablemodulesandmodpcohomology}.

Our restriction on associated primes is the following. For finite groups  Okuyama \cite{okuyama2010remark} showed that one and three are equivalent, and the equivalence of two and three is a special case of a result of Kuhn \cite{kuhnprimitives} given below. 

\begin{Thm}[\ref{restrictionsonassociatedprimes}]
Let $G$ be a compact Lie group and $E<G$ a $p$-torus. The following are equivalent:
\begin{enumerate}
\item $E$ represents an associated prime in $H_G^*$.
\item $E$ represents an associated prime in $H^*_{C_G E}$.
\item The  depth of  $H^*_{C_G E}$ is $\rank E$.
\end{enumerate}
\end{Thm}
This theorem is useful because it allows one to rule out, without cohomology computations, some $p$-tori from representing associated primes. For example if the Duflot bound for depth for $H^*_{C_GE}$ is bigger than $\rank E$, then $E$ can't represent an associated prime in $H_G^*$. Our proof uses the Duflot filtration and goes through equivariant cohomology, and for particular groups where there is a good understanding of the relevant $S$-manifold, our methods could presumably give more restrictions on associated primes. 

Carlson conjectured that for finite groups there is always an associated prime of dimension equal to the depth of $H_G^*$. Green \cites{greencarlsonsconjecture} and Kuhn \cites{kuhnprimitives,kuhnnilpotence} have shown that for $p$-groups and compact Lie groups respectively, this is true when the Duflot bound for depth is sharp. We give in \ref{carlsonsconjecturewhenduflotboundissharp} a different proof that also applies to compact Lie groups.
\begin{Thm}[\ref{carlsonsconjecturewhenduflotboundissharp}, Kuhn \cites{kuhnprimitives,kuhnnilpotence}]
For $G$ a compact Lie group,  $\depth H_G^*$ is equal to the $p$-rank of the center of $G$ if and only if the maximal central $p$-torus of $G$ represents an associated prime.
\end{Thm}

% \subsubsection{Local cohomology modules for $p$-Sylows of $S_{p^n}$}
We also study the relationship between the posets  mentioned above associated to  manifolds $M, N$ with $M \to N$ an equivariant principal $K$-bundle for $K$ a finite group. In favorable situations there is an induced map between the posets that is an analog of a principal $K$-bundle in the category of posets. For certain group extensions (see section \ref{itrivialbundles} for definitions) this leads to a spectral sequence computing local cohomology,
% \begin{Thm}[\ref{spectralsequenceforitrivialextension}]
% If $1 \to H \to G \to K \to 1$ is $i$-trivial, then for any faithful representation $G \to U(V)$  there is a spectral sequence with $E_2^{p,q} \cong H^p(K, \HH^q F_i H_S^* H \backslash U(V)) $ converging to $\HH^{p+q} F_i H_S^* G \backslash U(V)$.
% \end{Thm}
which leads to computations involving the top local cohomology modules of the group cohomology of the Sylow $p$-subgroups of $S_{p^n}$, which we denote by $W(n)$. 

% In other word, if we let $W(1)=\ZZ/p$, then we can inductively define $W(n)$ by $W(n)=W(n-1) \wr \ZZ/p$. 

We compute the local cohomology of $W(n)$ in terms of the Tate cohomology of $W(n-1)$.
\begin{Thm}[\ref{computationoflocalcohomologyofH_SU(V)/W(n)}]\label{computationforHSWNinintro}
For $W(n) \hookrightarrow U(V)$  a faithful representation, let $k$ be the dimension of the manifold $U(V)$, and let $d$ be $-p^{n-1}+k$. For $ 0< i < p-3$, we have that $$\HH^{p^{n-1}-(p-3)+i} H_S^*( W(n) \backslash U(V))$$ is isomorphic as an $H^*_{W(n)}$-module to $$\widehat{H}^{i-(p-2)}(W(n-1), \Sigma^d(H_S^* E(n) \backslash U(V))^*). $$

For the top local cohomology, we have that $\HH^{p^{n-1}}( H_S^* W(n) \backslash U(V))$ is isomorphic as an $\FF_p$-vector space to $$\widehat{H}^{-1}(W(n-1), \Sigma^d(H_S^* E(n) \backslash U(V))^*) \oplus N(\Sigma^d(H_S^* E(n) \backslash U(V))^*).$$ Here $N$ is the norm map, and $\widehat{H}$ is Tate cohomology. 
\end{Thm}

Local cohomology is known to vanish in degrees below the depth and above the dimension, and it is nonvanishing at degrees equal to the  depth and dimension, and at any degree equal to the dimension of an associated prime. 

The regularity theorem concerns the vanishing of $\HH^{i,j}H_G^*$ for $j>-i$, but  little  is known about the vanishing and nonvanishing of the entire $\HH^i(H_G^*)$ beyond what is known for general graded rings. In particular, it is not known if there is any finite group $G$ with $\depth H_G^*=d$ and $\dim H_G^*=r$ and some $i$ with $d<i<r$ with $\HH^i H_G^*=0$. In other words, in the range where local cohomology can be either zero or nonzero, it is not known if there is a group where local cohomology is ever nonzero.
Perhaps no such group exists. We can compute enough of the local cohomology of $H^*_{W(n)}$ to show that that top $p-2$ local cohomology modules for this ring are nonzero.
\begin{Thm}[\ref{abandofnonvanishinglocalcohomology}]
For $0 \le i <p-3$, $\HH^{p^{n-1}-i}(H^*_{W(n)}) \not= 0$.
\end{Thm}

This shows that there are arbitrarily long intervals where local cohomology is nonzero and where there is no a priori reason (in light of the fact noted above about associated primes) for local cohomology to be nonzero (the embedded primes for $H^*_{W(n)}$ are all of rank less than $p^{n-1}-(p-3)$).
\begin{Cor}[\ref{mysteriousgrap}] For each $p \ge 5$, for each $n$ there exists a $p$-group $G$ and an $i$ so that $\HH^{i+j}(H_G^*)\not=0$ for all $0<j<n$ , and so that $i+j$ is not the dimension of an associated prime.
\end{Cor}

Additionally, the strong form of Benson's regularity conjecture is still open. This conjecture states that $\HH^{i,-i}(H_G^*) =0$ for $i\not= \dim H_G^*$. It is known by the work of Symonds that this can only fail for $i$ equal to the rank of a maximal $p$-torus in $G$, so it is already known that this conjecture is true for $W(n)$ in the range which we compute in \ref{computationoflocalcohomologyofH_SU(V)/W(n)}. Nevertheless, it is interesting to observe that for the $i$ in our range the largest $j$ for which $\HH^{i,j}H^*_{W(n)} \not=0$ is  smaller than predicted by the strong form of the regularity conjecture.

In fact, we have the following:
\begin{Cor}[\ref{themysterioustriangleofzeroes}]
For $0 <  i < p-3$, $\HH^{p^{n-1}-i,j}H^*_{W(n)}=0$ for $j>-p^{n-1}$. 
\end{Cor}

In other words, the strong regularity conjecture, which is confirmed for these groups in this range, tells us that $\HH^{i,j}H^*_{W(n)}$ should be $0$ for $j>-(i+1)$, but in fact we can show in this range the bound can be improved to $j>-p^{n-1}$, independent of $i$.

\subsection{Outline}

In section \ref{sectiononabstractstuff} we study an abstract version of the Duflot filtration in equivariant cohomology. The most basic definitions and consequences are in \ref{moduleswithafreerankfiltration}. The refinement of the filtration is in \ref{stratifiedfiltrations}, and consequences of the refinement are in \ref{moreresultsforstratifiedfiltrations}. We study group actions on the refining poset and maps between filtered modules in \ref{KDuflot}, which leads to a spectral sequence computing local cohomology.

We apply the abstract results of section \ref{sectiononabstractstuff} to study equivariant cohomology rings in section \ref{chapteronequivariantcohomology}.

In section \ref{sectiononthecohomologyofBG} we apply the results earlier in the paper to the cohomology rings of the classifying spaces of compact Lie groups. We use the results of \ref{KDuflot} to study the local cohomology of the $p$-Sylow subgroups of $S_{p^n}$ in \ref{wreathproducts}.

\subsection{Acknowledgments}
This is the principal part of my thesis from the University of Washington, and I thank my advisor Steve Mitchell for introducing me to the subject and for sharing innumerable insights; this paper is dedicated to him.

I wish to thank my advisors John Palmieri and Julia Pevtsova for their support and for providing many helpful comments on drafts of this document.

\subsection{Notation and conventions}\label{notation}
Throughout, we fix a prime $p$.

We will denote by $P_W$ the polynomial algebra $S(W^*)$, where $W$ is a fixed $i$-dimensional $\FF_p$-vector space. The grading on $P_W$ is inherited from that on $W$, which is concentrated in degree $-1$ when $p=2$ and in degree $-2$ otherwise. So, after choosing a basis for $W$, $P_W=\FF_p[y_1,\dots ,y_i]$, where $|y_i|$ is either 1 or 2. For $V \subset W$, we denote by $P_V$ the symmetric algebra $S(V^*)$, graded in the same manner as $P_W$. We have a $P_W$ module structure on $P_V$ induced by the inclusion $V \to W$.

All rings are graded commutative and all modules are graded, and $\Sigma^d$ denotes the suspension functor. All homological algebra is to be done in the graded sense. 

We use $\HH^*_{\m}$ to denote local cohomology with respect to $\m$. All our local cohomology modules will be taken with respect to the maximal ideal of positive degree elements in an obvious algebra, so we will often omit $\m$. For a treatment of local cohomology that includes the graded case, see \cite{brodmannandsharp}. 

A $p$-torus is a group isomorphic to $(\ZZ/p)^r$ for some $r$, and if a $p$-torus $E$ is isomorphic to $(\ZZ/p)^r$ we say that the rank of $E$ is $r$.

For $X$ a $G$-space, we write $H_G^*X$ for $H^*(EG \times _G X;\FF_p)$, and write $H_G^*$ for $H_G^* pt$, which is $H^*(BG; \FF_p)$.

For $V$ a complex unitary representation of $G$, we will use $U(V)$ to denote the group of unitary isomorphims of $V$, which is of course equipped with a map $G \to U(V)$. If $V$ is equipped with a direct sum decomposition, i.e. $V \cong V_1 \oplus V_2$ as $G$-vector spaces, we will require that $U(-)$ respects this direction sum decomposition, so $U(V_1 \oplus V_2)=U(V_1) \times U(V_2)$. 

This is so we can refer to a map $G \hookto U(V)$, where we more properly mean a map $G \hookto \prod_i U(n_i)$. Generally any faithful representation of $G$ suffices for our theory, but there are a few points where we want to specify a representation $V$ of $G$ so that the center of $G$ maps to the center of $U(V)$, which is only possible in general if $U(V)$ is allowed to denote a product of unitary groups.

\section{An abstract treatment of modules with free rank filtrations}\label{sectiononabstractstuff}
In this section we give an abstract description of the structure that is present on the $S$-equivariant cohomology of a smooth manifold. We then show that whenever this structure is present there are analogs of Quillen's theorem on dimension, Duflot's theorem on depth, Carlson's detection theorem, and Symonds' regularity theorem. Then in \ref{KDuflot} we describe how the structures studied earlier in the section behave with respect to certain classes of maps, which leads to a spectral sequence for local cohomology modules.

\subsection{Modules with free rank filtrations}\label{moduleswithafreerankfiltration}
\subsubsection{Motivation from topology}
Our initial goal is to develop a framework for studying the algebraic consequences of the Duflot filtration. To motivate the definitions used in this framework, we first recall some of the structure apparent in the Duflot filtration. We will return to this in much more detail in \ref{propertiesofHSM}.

Let $M$ be a smooth $S$-manifold, where $S$ is a $p$-torus. Then Duflot in \cite{duflotfiltration} defines a filtration on $H_S^*M$ so that $F_i/F_{i+1}$ is a sum of modules of the form $\Sigma^d (H^*_A \otimes H^*N)$, where $N$ is a manifold and $A$ is a rank $i$ $p$-torus. This $H^*_A \otimes H^*N$ arises as the $S$-equivariant cohomology of a submanifold $L$ of $M$ on which $A$ acts trivially and $S/A$ acts freely, so it is an $H_S^*M$-module via the restriction map $H_S^*M \to H_S^*L \cong H_A^* \otimes H^*N$.

Because $A$ acts trivally on $L$, there are $S$-equivariant maps:
$
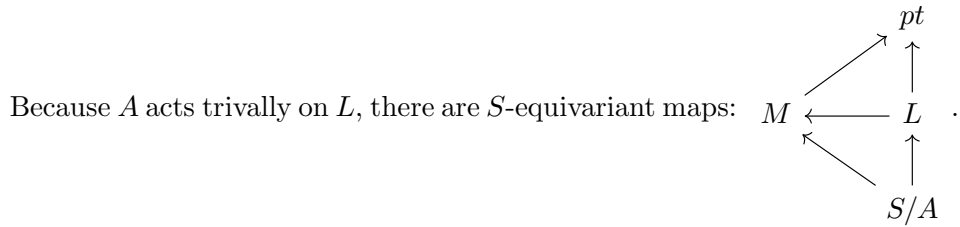
\begin{tikzcd}
& pt \\
M \urar & L \lar \uar \\
& S/A \ular \uar
\end{tikzcd}$.
Taking $S$-equivariant cohomology gives:
$
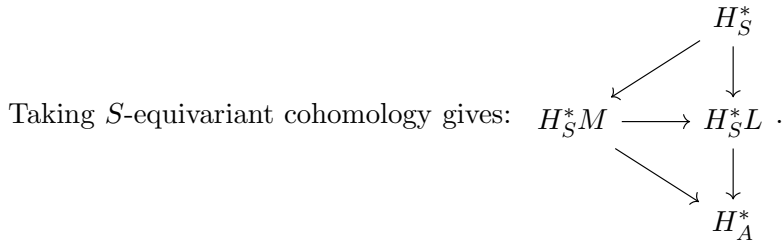
\begin{tikzcd}
& H_S^* \dar \dlar  \\
H_S^*M \rar \drar  & H_S^* L \dar \\
& H^*_A
\end{tikzcd}$.

When $p=2$, the top and bottom objects of this diagram are already polynomial rings, and when $p$ is odd there is an inclusion of a polynomial ring $P_S$ on $\rank S$ many variables  of degree two into $H_S^*$, and a map from $H_A^*$ onto a polynomial ring $P_A$ on $\rank A$ variables in degree $2$. For $p$  odd, we can write $H_S^*L$ as  a polynomial ring on $\rank A$ variables tensored with the tensor product of $H^*N$ and an exterior algebra on $\rank A$ variables, so as $P_A \otimes B$, where  $B$ is bounded as a graded ring.  So, for both $p$ odd and for $p$ even we have a diagram:

$\begin{tikzcd}
& P_S \dar \dlar  \\
H_S^*M \rar \drar  &P_A \otimes  B \dar \\
& P_A
\end{tikzcd}$.

The map $P_S \to P_A$ is induced by the $\FF_p$-linear inclusion $A \to S$.  

% For our description of these polynomial rings, recall that for $p=2$ the cohomology of a $2$-torus $A$ is naturally $S(A^*)$ where $A^*$ is concentrated in degree $1$, and for $p$ odd the cohomology of a $p$-torus $B$ is naturally $S(B^*) \otimes S(\Sigma B^*)$, where $B^*$ is in degree $1$ and therefore $\Sigma B^*$ is in degree $2$ (and symmetric algebras are to be taken in the graded commutative sense). Note that this is compatible with our grading conventions.

Suspensions of modules of the  form of $H_S^*L$ we will call ``$r$-free'' where $r$ is the rank of $A$. We will now define this in a purely algebraic setting.
\subsubsection{Free rank filtrations}

Fix an $\FF_p$-vector space $W$.

We are interested in Noetherian $\FF_p$-algebras $R$ that are also finite algebras over $P_W$, i.e. finitely generated  $\FF_p$-algebras $R$ with a (graded) $\FF_p$-algebra map $P_W \to R$ making $R$ into a finitely generated $P_W$-module. Recall from \ref{notation} that $P_W$ is polynomial on $\rank W$ variables which are in degree one when $p$ is $2$, and in degree two when $p$ is not $2$.

Throughout this section, $R$ will be a fixed finite $P_W$-algebra.

\begin{Def}An $R$-module $M$ is called \emph{$j$-free} if it is a suspension of an $R$-algebra of the form $ P_V \otimes N$, where $V$ is a $j$-dimensional subspace of $W$ and $N$ is a bounded connected $\FF_p$-algebra.

Moreover, in the commuting diagram of algebra maps:
\begin{tikzcd}
& P_W \dar \arrow{dl} \\
R \rar \drar & P_V \otimes N \dar \\ 
& P_V
\end{tikzcd}
we require that the map $P_W \to P_V$ is induced by the inclusion $V \to W$ (the map $P_V \otimes N \to P_V$ is the obvious one, induced by the identity $P_V \to P_V$ and the augmentation $N \to \FF_p \to P_V$).
\end{Def}
% The name is indicating that as $P_W$-modules, $j$-free modules are in particular the pullback of a free $P_V$-module under{} the map $P_W \to P_V$.

\begin{Def}
A descending $R$-module filtration $0=F_{i+1} \subset F_i \subset F_{i-1} \subset \dots \subset F_0=L$ of an $R$-module $L$ is called a \emph{free rank filtration} if $F_{j}/F_{j+1}$ is a finite sum of  $j$-free modules.
\end{Def} 
We will write $F_{j}/F_{j+1}$ as $\bigoplus_{V\subset W, \dim V=j} \bigoplus_{k=0}^l \Sigma^{d_{V,k}}(P_{V,k} \otimes N_{V,k})$, and we denote $\Sigma^{d_{V,k}}P_{V,k}\otimes N_{V,k}$ as $M_{V,k}$. We will denote the map $R \to P_{V,k}$ as $\phi_{V,k}$, but whenever it is possible to do so without confusion we will omit the $k$ in the subscript.  The direct sum decomposition of $F_j/F_{j+1}$ is part of the data of a free rank filtration, and the same $j$-dimensional subspace of $W$ can occur more than once in this direct sum decomposition.

%In fact the same map $\phi _V: R \to P_V$  can occur more than once in the decomposition of $F_j/F_{j+1}$, so we can decompose as $\oplus P_V \otimes ( \oplus_{\text{all $\phi_{V,k}$ are equal}} N_{V,k} )$. In this decomposition of $F_j/F_{j+1}$, we will write each $\oplus_{\text{different} P_{V,k} } P_V \otimes ( \oplus_{\text{all $\phi_{V,k}$ are equal}} N_{V,k} )$ as $M_{V,l}$ or just $M_V$, and $\oplus_{\text{all $\phi_{V,k}$ are equal}} N'_{V,k} $ as $N_{V,l}$, or just $N_V$. Fortunately in many examples $N_V$ and $M_V$ are equal to $N_V$ and $M_V$ respectively. 

\begin{Def}In the sequel we will want to add the condition that the kernels of the  $\phi_{V}: R \to P_{V}$ are distinct; we will call a free rank filtration with this extra property \emph{minimal}.
\end{Def}

\begin{prop}\label{depthanddimoftruncations}
If an $R$-module $L$ has  a free rank filtration, then the Krull dimension of  $L/F_s $ is less than $s$. If $R$ is in addition connected, and $F_{s+1} \not= F_s$, then $\depth F_s \ge s$.
\end{prop}
% \begin{remark}
% This is an analog of Quillen's theorem on dimension and the Duflot bound for depth, and we will later use this theorem to derive these theorems in group cohomology.
%\end{remark}
\begin{proof}
The proofs of both statements follows from the induced filtrations on $L/F_s$ and $F_s$, and the behavior of depth and dimension on short exact sequences.

% For the first statement, consider the filtration $F_s/F_s \subset F_{s-1}/F_s \subset \dots F_0/F_s=L/F_s$. The subquotients in this filtration are $F_{s-1}/F_s, F_{s-2}/F_{s-1}, \\ \dots, F_0/F_1$. The Krull dimension of $F_s/F_{s+1}$ is less than or equal to $s$ (and in fact equal to $i$ unless $F_s=F_{s+1}$ and $F_s/F_{s+1}=0$), so as $\dim L/F_s= \max_{k \le s} \dim F_{k-1}/F_k$, the result for $L/F_s$ follows.

% For the second statement, we consider the filtration of $F_s$ given by $F_i \subset F_{i-1} \subset \dots \subset F_s$. The subquotients of the filtration are $F_{i-1}/F_{i}, \dots , F_{s}/F_{s+1}$. The subquotient $F_{i}/F_{i+1}$ has depth $i$, so as the depth of a filtered module is greater than or equal to the minimum depth of the subquotients, the depth must be greater than or equal to $s$. 
The assumption that $R$ is connected is just because it is traditional to only define depth in the context of modules over local rings, and if $R$ is connected it is local in the graded sense.

\end{proof}
\begin{cor}
If $R$ is connected and $L$ has a free rank filtration, then the depth of $L$ is greater than or equal to the smallest $k$ such that $F_{k+1} \not= F_k$.
\end{cor}
\begin{proof}
This follows from Lemma \ref{depthanddimoftruncations}, because for such a $k$, $F_k=F_0=R$.
\end{proof}

\begin{Def}\label{algebraicdefinitionoftoral}
 A prime $\p$ of an algebra $R$ with a free rank filtration is called \emph{toral} if it is  the kernel of one of the  maps $\phi:R \to P_{V}$ appearing in the free rank filtration.
\end{Def}

\begin{prop}\label{algebraictoraltheorem}
If $L$ has a free rank filtration, then every element of $\ass_R L$ is toral.
\end{prop}
\begin{proof}
We have that $\ass_R L \subset \bigcup_{V,k} \ass_R M_{V,k}$, so it suffices to prove that each element of $\ass_{R} M_{V}$ is toral. The map $ \ass_{M_V} M_{V} \to \ass_R M_{V}$ is surjective.  

But $M_V$ is a domain tensored with a finite dimensional (as an $\FF_p$-vector space) algebra, and any such algebra has a unique associated prime.
So the only associated prime of $M_{V}$ as a module over itself is the kernel of the map $M_V \to P_V$, and we have our result.
\end{proof}
%\textbf{Note:} This proof holds in slightly more generality than stated, it doesn't use that $R$ is an algebra over $P_W$, and that consequently each $P_V$ is finite over $R$.

\begin{prop}\label{duflotcomplex}
If $R$ is a connected $P_W$ algebra and $L$ has a free rank filtration, then there is a cochain complex $DL$ where $$DL^j=\bigoplus_{V\subset W, \dim V=j} \left(\bigoplus_{k=0}^l \Sigma^{-\sigma_j}( \Sigma^{d_{V,k}}(P_{V,k}^* \otimes N_{V,k})) \right),$$ where $\sigma_j=j$ when $p=2$ and $\sigma_j=2_j$ when $p$ is odd, and $(-)^*$ denotes the graded linear dual, and this cochain complex has the property that $H^i( DL)= \HH^i(L)$, the $i^{th}$ local cohomology of $L$ as an $R$ module with respect to the maximal homogeneous ideal $\m$ of positive degree elements.

We call this chain complex the Duflot  complex of $L$, and this  complex is functorial for module maps preserving the filtrations.
\end{prop}
% \begin{remark}
% In other words, the $i^{th}$ term of the Duflot complex is a sum of shifts of modules of the form of the dual of polynomial ring in $i$ variables, tensored with a bounded module.
% \end{remark}

\begin{proof}
We have a diagram: 

$\begin{tikzcd}
& L \dar \\
F_{i-1}/F_i \rar & L/F_i \dar \\
& \vdots \dar \\
F_1/F_2 \rar & L/F_2 \dar \\
F_0/F_1 \rar & L/F_1 \dar \\
0 \rar & L/F_0=0
\end{tikzcd}
$

Here, a horizonal arrow followed by a vertical arrow is a short exact sequence. All maps are maps of $R$-modules, so we can apply the local cohomology functor to get an exact couple:

$\begin{tikzcd}
& \HH^*L \dar \\
\HH^*F_{i-1}/F_i \rar &\HH^* L/F_i \dar \\
& \vdots \ular[dotted] \dar \\
\HH^* F_1/F_2 \rar & \HH^* L/F_2 \dar \\
\HH^* F_0/F_1 \rar & \HH^* L/F_1 \ular[dotted] \dar \\
0 \rar & \HH^* L/F_0=0 \ular[dotted]
\end{tikzcd}
$

Here the dotted maps have cohomological degree $1$, and are the boundary maps in the local cohomology long exact sequence. This exact couple gives a spectral sequence with $E_1^{t,q}= \HH^{t+q} F_t/F_{t+1} $ converging to  $\HH^*L$. 

However, $F_t/F_{t+1}=\bigoplus_{V\subset W, \dim V=t} \left(\bigoplus_{k=0}^l  \Sigma^{d_{V,k}}(P_{V,k} \otimes N_{V,k}) \right)$, therefore $\HH^{t+q}F_t/F_{t+1}$ is a sum of modules of the form $\HH ^{t+q}(\Sigma^{d_{V}} P_{V} \otimes N_{V})$ where $\dim V=t$, which we now compute.

The independence theorem for local cohomology, see  4.2.1 of \cite{brodmannandsharp}, tells us that we can compute the local cohomology of $P_{V} \otimes N_{V}$ either over $R$ or over $P_{V}$.

By the K\"unneth theorem for local cohomology and the fact that $N_V$ is bounded in degree, $\HH^*(P_{V} \otimes N_{V})= \HH^*P_{V} \otimes N_{V}$.  %As $N_{V}$ is bounded as a graded module its local cohomology with respect to the maximal ideal of $P_{V}$ is zero in degrees abov{}e $0$ and $N_V$ in degree $0$. Finally, we recall that $\HH^i(P_{V}) \cong \Sigma^{-\sigma_j} (P_{V})^*$ when $i=p$, and $0$ otherwise; where $\sigma_j$ is as defined in the theorem.
Therefore, since $\HH^*P_{V}$ is
$\Sigma^{-\sigma_{t}}(P_V^*)$ in degree $t$ and zero otherwise,
the spectral sequence collapses to the bottom row at the $E_1$ page and we have our result.

Functoriality follows from a map of filtered modules inducing a map on the exact couples giving the spectral sequence.

\end{proof}
\begin{Def}
For $R$ a local Noetherian graded ring and $L$ a finitely generated module over $R$, let $a_iL= \sup_j \{ \HH^{i,j} L \not=0\}$. Then recall that the \emph{Castelnuovo-Mumford regularity} of $L$, or $\reg L$, is $\sup_{i}\{ a_i L +i \}$.
\end{Def}

\begin{cor}\label{regularitytheorem}
Let $t(N)$ denote the top nonzero degree of a bounded module $N$. If $L$ satisfies the hypotheses of the previous propoistion, then $\reg L \le \max \{ t(N_{V,k})+d_{V,k} \}$ when $p=2$, and $\reg L \le \max \{ t(N_V)+d_{V,k}-\rank(V)\}$ when $p \not=2$.
\end{cor}
\begin{proof}
This comes from examining the Duflot chain complex. Under the hypotheses of the theorem $DL^{i,j}$ is zero for $j> -\sigma_i+ \max  \{ t(N_{V,k}) +d_{V,k}\}$, giving the result.
\end{proof}

Finally, we give the dual formulation of \ref{duflotcomplex}. For a connected $P_W$ algebra $R$, let $\mathfrak{d}$ denote the Matlis duality functor. In our setting this coincides with $\Hom_{\FF_p}(-,\FF_p)$, see \cite{brodmannandsharp} exercise 14.4.2.

\begin{prop}\label{matlisdualcomplex}
For $R$ a connected $P_W$ algebra and $L$ a module with  a free rank filtration, taking the Matlis dual of $DL$ gives a chain complex with $$\mathfrak{d}(DL^j)= \bigoplus_{V\subset W, \dim V=j} \bigoplus_{k=0}^l \Sigma^{\sigma_j}(\Sigma^{-d_{V,k}}P_{V,k} \otimes N_{V,k}^*),$$ and the homology of this chain complex is the Matlis dual of $\HH^*L$.
\end{prop}
In other words, the $i^{th}$ term of the dual chain complex is a sum of shifts of polynomial rings tensored with bounded modules.

\subsection{Modules with a stratified free rank filtration}\label{stratifiedfiltrations}
\subsubsection{Preliminary definitions and motivation from topology}
In the application we are focused on, the filtration on $H_S^*M$ will be related to the fixed points of subgroups of $S$. In fact, these submanifolds refine the filtration and clarify how we can   understand $H_S^*M$ by studying the different $H_S^*Y$, for $Y$ a component of the fixed points of a subgroup of $S$. In this section we  give an algebraic description of how a free rank filtration can be refined to a filtration by a poset.

First, we describe some of the algebraic structure present when  considering embedded manifolds along with their pushforward and restriction maps.

\begin{Def}
We say that a $P_W$-algebra $T$ is \emph{embedded} in $R$ with codimension $d_T$ if there is a map of $P_W$ algebras $R \to T$ and a map of $R$-modules $\Sigma^{d_T} T \to R$, so that the composition $\Sigma^{d_T} \to R \to T$ is a map of $T$ modules. We call the map $i_*:\Sigma^{d_T} T \to R$ the \emph{pushforward}, and the map $i^*:R \to T$ \emph{restriction}. We define the \emph{Euler class} of $T$, or $e_T$, to be $i^*i_*1$. 
% We will also refer to $i_*$ as $\tr$ or \emph{transfer} or also as the Gysin map, and to $i^*$ as \emph{restriciton}, $\res$.
\end{Def}

Note that if $T$ is embedded in $R$, the composition of pushforward followed by restriction is multiplication by $e_T$.

The terminology comes from topology, if $X \to Y$ is an embedding of smooth manifolds where $X$ has codimension $d$, then $H^*X $ is embedded in $H^*Y$ with codimension $d$. This is also true equivariantly. %Note that in the example on an embedded smooth manifold, the composition of the pushforward and restriction is a map $H^*Y$ modules, but we are not assuming this.

We are only interested in embedded $P_W$-algebra in the case that the Euler class is not a zero divisor; we give a name to this class of algebras.

\begin{Def}
If $T$ is embedded in $R$  and $e_T$ is a nonzero divisor, then we say that $T$ is \emph{fixed} in $R$.
\end{Def}
The terminology comes from equivariant topology, if $M$ is a smooth $S$-manifold, for $S$ a $p$-torus, and $Y$ is an orbit of a connected component of the fixed point set of a sub-torus of $S$, then $H_S^*Y$ is fixed in $H_S^*M$.

Note that if $T$ is fixed in $R$, then the pushforward $\Sigma^{d_T} T \to R$ is injective.

\subsubsection{Filtrations by posets}
The free rank filtrations that occur in topology have refinements to filtrations by posets. Here we discuss some features of the posets we are interested in.

Let $P$ be a poset weakly  coranked by the natural numbers, i.e. equipped with  a map of posets $r: P \to \NN^{op}$. 

Given such a poset, there are a few related posets: $P_{\ge j}$ is the subposet of elements of rank greater than or equal to $j$, and $P_{<Y}$ for $Y \in P$ is the subposet of elements less than $Y$, and similarly $P_{\le Y}$, $P_{> j}$, etc. We'll use $P_j$ for the set of elements of rank $j$.

A filtration of a module $L$ by such a poset $P$ is a submodule $F(X)$ for each $X \in P$, and where if $X \le Y$, $F(X)\subset F(X)$, and so that $\sum_{X \in P} F(X)=L$. 

% We want to think of such a filtration as a functor from $P$ to the category of $R$-modules $N$ equipped with an injection $N \to L$   (where $L$ is the module being filtered), so from now on we write $F(X)$ for $F_XL$. 

We will continually be referring to the lowest subquotient of $F(X)$ in the filtration on $F(X)$ by $P_{\le X}$, so if $r(X)=j$, let $$\gr[j] F(X)= F(X)/ \sum_{Y \in P_{<X}} F(Y).$$

Given a module filtered by $P$, there is an associated filtration by $\NN$, where $F_j=\sum_{X \in P_{\ge j }} F(X)$. 
\begin{Def} The filtration of $M$ by $P$ is called \emph{good} if for all $j$ the map $\oplus_{Y\in P_j}\gr[j] F(Y) \to F_j/F_{j+1}$ is an isomorphism.
\end{Def}
All the filtrations we are interested in are good, so the associated graded can be computed one element of the poset at a time. The point is that a filtration by a weakly coranked poset gives a filtration by $\NN$, so there are two associated gradeds: one indexed by $\NN$ and one indexed by the poset. When the filtration is good, these associated gradeds agree. 
% \begin{remark}
% For a trivial example where the filtration is not good, consider a filtration on $\ZZ$ by the poset $a \from b \to c$, where $b$ has rank $1$ and $a,c$ have rank $0$, and filter by $f(b)=0$ and $f(a)=f(c)=\ZZ$. Then the associated filtration by $\NN$ just has $F_0=\ZZ$ and $F_1=0$, and we see that the filtration is not good, since $f(a)/f(b) \oplus f(c)/f(b)=\ZZ \oplus \ZZ$, which of course is not $F_0/F_1$.
% \end{remark}%For an example of a filtration that is not good, take
%put in this example! 

 There is a potential for confusion between the filtration by $P$ and the induced filtration by $\NN$, so we will use letters like $X,Y$ for elements of $P$ and $i,j,l$ for natural numbers. 
 % We will use capital letters for elements of $P$ and natural numbers will always be lower case.

%\textbf{Note:} The definition of the induced filtration by $\NN$ requires that $F_j$ is sum to be over everything of rank greater than or equal to $j$ rather than just everything equal to $j$ to ensure that the filtration is still defined for all $\NN$ and so that $F_0=M$. It may be the case that sometimes a poset we are using to filter has the property that all elements have rank greater than some nonnegative number.

Note that if the filtration of $L$ by $P$ is good, then for each $X \in P$, $F(X)$ has a good filtration by $P_{\le X}$. The fact that there is a filtration is automatic; to see that it is good consider the following commutative  diagram.

\begin{tikzcd}
\bigoplus_{Y \in P_j} \gr[j]F(Y)  \dar & \bigoplus_{Z \in (P_{\le X})_j} \gr[j]F(Z) \lar \dar \\
F_j/F_{j+1}L & F_j/F_{j+1} F(X) \lar
\end{tikzcd}

 The right hand arrow is always surjective (the map that is required to be an isomorphism in a good filtration is always surjective), so because the left hand arrow is injective, the righthand arrow must be injective also, and therefore an isomorphism.

\subsubsection{Stratified rank filtrations}
\begin{Def}An $R$-module $L$ has a \emph{free rank filtration stratified by a weakly coranked finite poset $P$} if there is a good filtration of $L$ by $P$ such that for all $j$  each $\gr[j]F(Y)$ is $j$-free.
\end{Def}

Note that the induced filtration by $\NN$ of such a stratified free rank filtration is a free rank filtration, where the decomposition of $F_j/F_{j+1}$ into a sum of $j$-free modules is induced by $F_j/F_{j+1}=\oplus_{Y \in P_j}\gr[j]F(Y)$. Also note that for each $Y \in P$, $F(Y)$ has a stratified free rank filtration by $P_{\le Y}$. The condition that our poset is finite ensures that our associated filtration by $\NN$ is finite.

\begin{Def}\label{thedefinitionoftopologicalstratifications}
If $L$ is an ideal of $R$, a free rank filtration stratified by a finite poset $P$ is called a \emph{topological} filtration stratified by $P$ if the following are satisfied:
\begin{enumerate}
\item For each $Y \in P$, $F(Y)=i_*(\Sigma^{d_T} T)$ for an embedded, fixed $T$.  We use the same symbol for $T \in P$ and the  embedded $T$ corresponding to $F(T)$.
\item For $T \in P$, the filtration  on $L$ induces on $T$  a free rank filtration (of $T$-modules) stratified by $P_{\le T}$ via the filtration on $\Sigma^{d_T}T \subset L$ obtained by restricting the functor from $P$ to $P_{\le T}$, and the structure maps $T \to P_V$ for the various $j$-free modules occurring as subquotients give us a commuting diagram:

\begin{tikzcd}
R \dar \drar & \\
T \rar & P_V
\end{tikzcd}.

\item For all $U<T$ in $P$ with corresponding embedded $U$ and $T$ in $R$, $U$ is also embedded in $T$, and the  composition of the two pushforwards $\Sigma^{d_{T,R}}( \Sigma^{d_{U,T} } U \to T) \to R$ is the pushforward of $\Sigma^{d_{U,R}} U \to R$.
\end{enumerate}
\end{Def}

In the third point, the notation $d_{U,T}$ refers to the codimension of $U$ in $T$.  Condition three implies that the codimension of $U$ in $T$ plus the codimension of $T$ in $R$ is the codimension of $U$ in $R$. Note that if $L \subset R$ has a free rank filtration that is topologically stratified by $P$, then each embedded $T$ appearing in the stratification for $L$ has a free rank filtration topologically stratified by $P_{\le T}$.

Here we record some useful properties of such filtrations. 
\begin{prop}
If $L$ has a topological filtration stratified by $P$, then the following hold for all $T \in P$. We use the same symbol $F$ for the functor appearing in the filtration of $L$ by $P$ and for the functor appearing in the filtration of the embedded $T$ by $P_{\le T}$.
\begin{enumerate}
\item For $j=r(T)$ we have that $F_{j} T=F_0 T=T$, so $F_{l}/F_{l+1}T=0$ for $l <j$.

\item Each $F_l/F_{l+1}\Sigma^{d_T} T \to F_l/F_{l+1} L$ is the inclusion of direct summands.
\item For each $M_{V}$ appearing in the decomposition of $F_j /F_{j+1} L$ as $j$-free modules (so $\dim V=j$), there is a unique $T \in P_j$  such that $\gr[j]F(T)$ maps isomorphically onto $M_{V}$.
\end{enumerate}
\end{prop}
\begin{proof}
\begin{enumerate}
\item Recall that $F_l T=\sum_{X \in (P_{\le T})_{\ge l}} F(X)$. So if $l \le  j$, then $T$ occurs in this sum and $F_lT=T$.
\item For $T \in P$, we consider $F_l/F_{l+1} T$, which because the original filtration is good, is $\oplus_{ U \in (P_{\le y})_l }\gr[l] F(U)$. But each such $U$ is also embedded in $R$, so on each component of $F_l/F_{l+1}T$ our map is part of the composition $\Sigma^{d_{Y}}(F_l/F_{l+1} \Sigma^{d_{U,Y}} U \to F_l/F_{l+1} Y) \to F_l/F_{l+1} L$, giving our result.
\item This follows from the fact that the filtration is good.
\end{enumerate}
\end{proof}
%\textbf{Note}: You also might expect the restriction maps $R \to T$ to be filtered with respect to the filtration by $\NN$. Indeed, in the example we have in mind, this is the case. However, we were not able to prove this from our definitions and we don't have a use for this fact, so we will leave this out for now.

Note that the stratification condition implies that the coproduct of all the $\Sigma^{d_T} T$ surjects onto $L$, because this map is surjective on associated gradeds. Actually, in many examples $L=R$, and $R$ itself will be one of the $T$s appearing in the topological stratification, corresponding to the maximal element of $P$.

However even in this situation when $R$ is one of the $T$s, we can restrict to those  elements of the poset that have rank equal to $l$, then by the definition of the associated filtration on $\NN$ the coproduct of all of these (with the appropriate shifts according to the codimension) surjects onto $F_{l}(R)$. 
In fact, we can look at the subposet $P_{\ge l}$, and we can compute $F_l$ as a colimit over this poset.

\begin{lemma}\label{welldefinedlemma}
Suppose $L$ has a topological filtration stratified by $P$ and that $Y,Z$ are embedded rings appearing in the stratification. Then, if $x$ is in the image of the pushforward from $Y$ and the pushforward from $Z$,  there are $W_t$s embedded in $R$, $Y$ and $Z$ as part of each stratification, and $w_t \in W_t$ so that $x= \Sigma_t (i_{W_t})_* w_t$, and so that for each $t$ there is a commuting diagram:

$
\begin{tikzcd} 
\Sigma^{d_{W_t}} W_t \rar \drar & \Sigma^{d_Y} Y\rar & L \\
& \Sigma^{d_Z} Z  \urar &
\end{tikzcd}$ .
\end{lemma}
\textbf{Note:} The subscript on $W_t$ is just an indexing, it has nothing to do with the various filtrations.
\begin{proof}
The proof is by downward induction on the filtration degree. First, consider the highest degree filtration that is potentially nonzero, $F_iL$, so suppose that $x \in F_iL$ is in the image of the pushforward from some $Y$ and from some $Z$, so there exist $y \in \Sigma^{d_Y}Y,z \in \Sigma^{d_Z}Z$ with $(i_Y)_*(y)=(i_Z)_*(z)=x$. Then $x \in F_{i}/F_{i+1}L=F_iL$ is equal to $x_1 +x_2 + \dots x_k$, where $x_j \in \gr[i]F(W_j)$,. Similarly, $y=y_1+\dots +y_k$, $z=z_1+\dots+z_k$, and each $y_j,z_j$ maps to $x_j$.

Then, because $Y$ and $Z$ are stratified, for each $j$ there is a $w_j \in \Sigma^{d_{W_j}}W_j$ so that $W_j$ is embedded in $Y$ and $Z$ and so that under the pushforward from $W_j$ to $Y$, $w_j \mapsto y_j$, and under the pushforward from $W_j$ to $Z$, $w_j \mapsto z_j$. That there is some $W_j$ with this property for $Y$ and $Z$ separately follows because $Y$ and $Z$ are stratified, but there is a common $W_j$ with this property because there is a unique $W_j$ appearing in the stratification for $R$ mapping to the summand supporting $x_j$. This also explains why the following diagram commutes:

$\begin{tikzcd}\Sigma^{d_{W_{j}}} W_j \rar \drar & \Sigma^{d_Y} Y \rar & L\\
& \Sigma^{d_Z}Z \urar &
\end{tikzcd}$.

Now, suppose the result is true for $F_l L$ with $l>j$, and suppose that $(i_Y)_*(y)=(i_Z)_*(z)=x$, where $x \in F_jL -F_{j+1}L$. Then $\overline{x} \in F_{j}L/ F_{j+1}L$ is equal to $x_1 + \dots + x_k$, where $x_l \in  \gr[j]Y_{V_{j,l}}$. By similar logic as in the preceding paragraphs, the images of $y,z$ in $F_jY/F_{j+1}Y$, $F_jZ/F_{j+1}Z$ can be written as a sum $\overline{y}=y_1+\dots+y_k$, $\overline{z}=z_1 + \dots + z_k$, where $y_l,z_l \mapsto x_l$. As $X,Y,Z$ are compatibly stratified, for each $l$ there is some  $w_l \in \Sigma^{d_{W_l}} W_l$ where $W_l$ is embedded in $Y,Z,R$ so that the pushforward of $w_l$ hits $y_l,z_l, x_l$ and so that the following diagram commutes: 

$\begin{tikzcd}\Sigma^{d_{W_{l}}} W_l \rar \drar & \Sigma^{d_Y} Y \rar & L\\
& \Sigma^{d_Z}Z \urar &
\end{tikzcd}$.

 But then $(i_Y)_*(y -\sum_l (i_{W_l})_* w_l)=(i_Z)_*(y- \sum_l (i_{W_l})_* w_l)=x- \sum_l (i_{W_l})_* w_l$. Additionally, $x- \sum_l (i_{W_l})_*w_l \in F_{j+1}L$, where our result is already assumed to be true, so we are done by induction.
\end{proof}

\begin{prop}
If $L$ has a  topological filtration stratified by $P$, then $\varinjlim_{T \in P_{\ge j}} F(T)=F_j L$.
\end{prop}
\begin{proof}
We show that $F_jL$ satisfies the required universal property. First, there are compatible maps $F(T) \to F_jL$ for $T  \in P_{\ge j}$. Now, suppose that there are compatible maps $F(T) \to[j_T] C$, where $C$ is some other $R$-module. We need to show that the dashed arrow in this diagram can be uniquely filled:

$\begin{tikzcd} F(T) \rar{i_T} \drar{j_T} & F_jL  \dar[dashed] \\
& C
\end{tikzcd}$. 

Because each map $F(T) \to F_jL$ is injective and because $\bigoplus_{T \in P_{\ge j}} F(T) \to F_jL$ is surjective by the definition of our filtration by $\NN$ on $L$, there is at most one map from $F_jL \to C$ making the diagram commute. We need to show that this map is well defined.

In other words, if $(i_T)_*(t)=(i_U)_*(u)$, we need to show that $j_T(t)=j_U(u)$. But this follows from Lemma \ref{welldefinedlemma}, because we will have a commuting diagram 
\begin{tikzcd}
& F(W_j) \arrow{dl}{i_1} \arrow{d} \arrow[bend left]{dd} \arrow{dr}{i_2} & \\
F(T) \arrow{r} \arrow{dr}{j_T} & L & F(U) \arrow{l} \arrow{dl}{j_U} \\
& C &
\end{tikzcd}
 and $w_j \in F(W_j)$ so that $t=\sum i_1(w_j)$, $u=\sum i_2(u_j)$. Then we see that $j_T(t)=j_U(u)= \sum j_{W_j} (w_j)$, which completes the proof.
\end{proof}

There is one final strong condition we will add to a free rank filtration topologically stratified by $P$.

\begin{Def}
An $R$-module with a  topological filtration stratified by $P$ is called \emph{Duflot} if for all $T \in P$, the corresponding embedded $T$ is isomorphic as $P_W$ algebras to $P_V \otimes T'$ where $T'$ is a $P_W$ algebra, and where $\dim V= r(T)$,  and $T/ \sum_{ X \in P_{<T}} X$ is $r(T)$-free and a suspension of a module of the form $P_V \otimes N$, and $T \to T/ \Sigma_{ X \in P_{<T}} X\cong \Sigma^d(P_V \otimes N)$ is induced by a  map $T' \to \Sigma^dN$. 

% We write a Duflot module as $(L,P,F)$, where $L$ is the module, $P$ is the filtering poset, and $F$ is the functor realizing the filtration.

%this definition can be cleaned up
\end{Def}
% For $X \in P$ as mentioned in \ref{thedefinitionoftopologicalstratifications} we just write $X$ for $F(X)$, so $T/ \sum_{ X \in P_{<T}} X$ means $F(T)/ \sum_{X \in P_{<T} } F(X)$.

We will refer to ``Duflot algebras'' and ``Duflot modules'' in the sequel: $R$ is a \emph{Duflot algebra} when $R$ itself has the structure of a Duflot module.

\subsection{Results on associated primes and detection for Duflot modules}\label{moreresultsforstratifiedfiltrations}

\begin{thm}\label{detectionofassociatedprimes}
If $L$ has a minimal topological filtration stratified by $P$, then $\phi_{V,k}$ represents an associated prime in $\ass_R L $ if and only if $\phi_{V,k}$ represents an associated prime in the embedded $T$ corresponding to $\phi_V$, as a module over itself.
\end{thm}

\begin{proof}
Let us unpack what this means.  Let $\phi_{V}: R \to P_V$ be one of the structure maps appearing the stratified free rank filtration for $L$. Associated to this $\phi_V$, there is a $\rank V$-free $M_V$ appearing as a direct summand of $F_{\rank V}/F_{\rank V+1} L$. By the stratification hypothesis there is a unique $T\in P$ whose highest filtered subquotient hits   $M_{V}$. So, we have a triangle $\begin{tikzcd} R \rar{\phi_V} \dar & P_V \\
T \urar{\phi_V}\end{tikzcd}$. The vertical map is the restriction map, and the maps going to the right are the two different structure maps, both of which we denote by $\phi_V$.

First, note that if $\phi_V$ represents an associated prime in $\ass_{T}T$, then considering the surjective map $\ass_{T} T \to \ass_R T$, $\ker \phi_V: T \to P_V$ is also an associated prime of $\ass_R T$. But then it is also an associated prime of $\ass_R \Sigma^{d_T} T$, and because there is an injective map $\Sigma^{d_{T}} T \to L$, it is an associated prime of $\ass_R L$.

Conversely, we  consider the short exact sequence $\Sigma^{d_{T}} T\to L \to \coker$. Any associated prime of $L$ as an $R$ module must be in $\ass_R \Sigma^{d_{T}} T$ or $\ass _R \coker$. We see that $\coker$ has a free rank filtration induced from those on $T$ and $L$, and moreover $\phi_V$ doesn't appear as one of the structure maps in the free rank filtration for $\coker$, as the filtered subquotients for $\coker$ are just those for $L$ modulo those for $\Sigma^{d_{T}} T$, and our filtration is minimal. So, if $\phi_V$ represents an associated prime in $R$, it must represent one in $T$ as well.

\end{proof}
In the above proof we crucially use the minimality of the free rank filtration, which tells us that we can check every potential associated prime on a unique $T$. However, we don't use the full strength of our embedding hypothesis: we just use that the pushforward is injective, not that the composition of pushforward and restriction is a map of $T$-modules.

Also note that under the hypotheses of the above theorem, by Lemma \ref{depthanddimoftruncations} the depth of $T$ is greater than or equal to $\dim V=i$, the lowest nonzero filtration degree. But $\ker \phi_{V}: T \to P_{V}$ is $i$ dimensional, so if $\ker \phi_{V}$ is associated in $T$, it is the smallest dimensional associated prime in $T$ and $\depth T=i$, as the dimension of an associated prime cannot be less than the depth.

Under the additional assumption that the filtration is Duflot, we can conclude the converse of this statement.

\begin{thm}\label{depthofcentralizers}
If $L$ is a Duflot module, then a  $\phi_V$ occurring in the filtration represents an associated prime in $\ass_R L$ if and only if the embedded $T$ corresponding to $\phi_V$ has depth equal to $\dim V$.
\end{thm}
\begin{proof}
We saw above that $\phi_V$ represents an associated prime in $\ass_R L$ if and only if $\phi_V$ represents an associated prime in $\ass_T T$, and that if $\phi_V$ represented an associated prime in $T$, then the depth of $T$ was $\dim V$.

For the converse, if $\depth T=\dim V$, then as $T=P_V \otimes T'$, by the K\"unneth theorem for local cohomology the depth of $T'$ as an module over itself must be zero, so there is some element $x \in T'$ annihilated by all of $T$, therefor $\ann_{T} 1 \otimes x=\ker(T \to P_V)$ is in $\ass_T T$.
\end{proof}

\begin{prop}\label{detection}
Suppose $R$  is Duflot algebra (so it has the required filtration as a module over itself). If $\depth R=d$, then $R \to[\prod_{res_T}] \prod_{T \in P_d } T$ is injective.
\end{prop}
We say that $R$ is \emph{detected} on such $T$
\begin{proof}
Suppose for a contradiction that $x$ is a nonzero element in the kernel of $R \to[\prod_{res_T}] \prod_{T \in P_d} T$, so $x$ restricts to zero on each $T$. We claim that $x$ is annihilated by $F_d$. To see this, it is enough to show that $x$ is annihilated by the image of each $\Sigma^{d_T} T$ with $r(T)=d$. So, suppose $i_*(t) \in \im (\Sigma^{d_T} T \to R)$, and consider $x i_*(t)$. But, because $i_*$ is a map of modules, this is $i_*i^*x t$, and therefore $0$ since $i^*(x)=0$.

Therefore $x$ is annihilated by $F_d$, so $F_d$ consists entirely of zero divisors, and is consequently contained in the union of all the associated primes of $R$, and therefore in one of the associated primes of $R$ by prime avoidance, so we have $F_d \subset \p$, where $\p$ is associated. But by \ref{depthanddimoftruncations} $\dim R/F_d < d$. Therefore $\dim R / \p <d$, and the depth of $R$ is less than $d$, a contradiction.

Here we are using the fact that the depth of a ring gives a lower bound for the dimension of an associated prime.
\end{proof} 

\subsection{K-Duflot modules and maps of Duflot modules}\label{KDuflot}
Sometimes, given an extension $1 \to H \to G \to K \to 1$ and a representation $G \hookto U(n)$, the associated bundle $K \to H \backslash U(n) \to G \backslash U(n)$ gives a close relationship between the poset controlling $H_S^* H \backslash U(n)$ and $H_S^* G \backslash U(n)$, perhaps after restricting to some higher level of the filtrations. Here we describe the algebraic structure this puts on the associated Duflot filtrations.

The main point of this section is Theorem \ref{spectralsequenceforKduflotmodules}. In order to make sense of this theorem we have to study how the map $F_i H_S^* G \backslash U(n) \to F_i H_S^* H \backslash U(n)$ mentioned above is controlled by a map between the stratifying posets and a natural transformation between the functors appearing in the filtration, and we need to similarly understand how the $K$-module structure on $H_S^* G \backslash U(n)$ is determined by an action of $K$ on the stratifying poset that is compatible with the filtering functor.

Before we  do this, we  note that most of the pleasant algebraic properties of a Duflot algebra $R$ are shared by each $F_i R$. 

\begin{lemma}\label{whytruncationsaregood}
If $R$ is a $P_W$-algebra that is Duflot as a module over itself with filtration stratified by $P$, then for all $i$:
\begin{enumerate}
\item $F_i R$ is a Duflot module, with filtering poset $P_{\ge i}$.
\item $\colim_{X \in P_{\ge i} } F(X)=F_i R$
\item The map $F_i R \to R$ induces an isomorphism on Duflot complexes in degree greater than or equal to $i$, and consequently an isomorphism on local cohomology in degree greater than $i$ and a surjection in degree $i$.
\end{enumerate}
\end{lemma}

In this section, we will be focusing on these $F_iR$.

Recall that we write a Duflot module as $(L,P,F)$, where $L$ is the module, $P$ is the filtering poset, and $F$ is the functor $P \to P_W$-mod realizing the filtration.

\textbf{Note}: The functor $F$ actually takes values in $R$-mod. However, we will want to consider maps between Duflot modules that come from different Duflot algebras, so we forget to $P_W$-mod.

\subsubsection{Morphisms of Duflot modules}
\begin{Def}
A map $P \to[\pi] Q$ of posets is a covering map if over every chain $C \subset Q$, there is a commutative diagram of posets
$\begin{tikzcd}
\pi^{-1}C \dar \rar{\sim} & C \times n \dlar \\
C
\end{tikzcd}$.
\end{Def}
By $n$ we mean the poset with $n$ objects and only identity arrows, in other words $C \times n$ is just $n$ copies of $C$.

\begin{Def}
A morphism of Duflot modules $(L,P,F) \to[\pi] (N,Q,G)$ consists of the following data:
\begin{enumerate}
\item A covering map of posets $\pi:P \to Q$
\item A natural transformation $G\pi \to F$.
\end{enumerate}
\end{Def}

\begin{lemma}\label{lemmaboutamorphisminducingmorphims}%%steve says this is a definition. Put in statement about uniquness 
A morphism $\pi: (L,P,F) \to (N,Q,G)$ induces a $P_W$-module map $\pi^*:N \to L$ which is uniquely characterized by: for all $Y \in Q$, the diagram:
$\begin{tikzcd}
G(Y) \rar \dar & N \dar \\
\prod_{X \in \pi^{-1}Y} F(X) \rar & L
\end{tikzcd}$ commutes.
\end{lemma}
\begin{proof}
We show how this information defines a map, and then uniqueness follows from the universal property of  colimits. 

Recall that $L= \colim_{X \in P} F(X)$ and that $N= \colim_{Y\in Q}G(Y)$. We define compatible maps $G(Y) \to L$ by $G(Y) \to \prod_{X \in \pi^{-1}Y} F(X) \to L$, where the left hand arrow is the product of the maps $G(\pi(X)) \to F(X)$ appearing in the natural transformation, and the right hand arrow is the sum of the inclusions.

To see that this gives a map from $N$ we must show that if we have a map $Y \to Y'$ in $Q$, then the diagram:
\begin{tikzcd}
G(Y) \dar \rar & G(Y') \dar \\
\prod_{X \in \pi^{-1} Y} F(X)  \dar \rar[dotted] & \prod_{X' \in \pi^{-1}Y'} F(X') \dlar  \\
N
\end{tikzcd} commutes, and this follows  because local triviality  ensures the existence of the dotted arrow.
\end{proof}

\begin{prop}
If $\pi:(L,P,F) \to (N,Q,G)$ be a morphism of Duflot modules, then the map on associated gradeds is a follows. Recall that $F_j/F_{j+1}L= \oplus_{X \in P_j} \gr[j] F(X)$, and that $F_j/F_{j+1}N=\oplus_{W \in Q_j} \gr[j] G(W)$. The map $\gr[j] G(W) \to \gr[j] F(X)$ is the one induced by the map $G(W) \to F(X)$ if $\pi(X)=W$ (the map appearing in the natural transformation), and zero if $W \not= \pi(X)$
\end{prop}
\begin{proof}
To determine $\gr[j] G(W) \to L$, we must study the composition:
$G(W) \to  N \to  L$.
This map factors as:
\begin{tikzcd}
\prod_{X \in \pi^{-1}W} F(X) \rar & L \\
G(W) \rar \uar & N \uar
\end{tikzcd}, so the map is as claimed when $\pi(X)=W$. 

For the other case, because the map factors as in the above diagram, if $x \in \im(G(W)) \cap F(Y)$ with $\pi_*(Y) \not=W$, then $x$ is in filtration degree greater than $j$, so we are done.
\end{proof}
\subsubsection{$K$-Duflot modules}
\begin{Def}\label{defofKduflot}
For $K$ a finite group, a \emph{$K$-Duflot module} is a Duflot  module $L$ with a left action of $K$ on $P$ by poset maps along with for all $k \in K$ natural isomorphisms $Fk \Rightarrow F$ so that for all $X \in P$, the induced map $Fk(X) \to F(X)$ is an isomorphism satisfying the following axioms:
\begin{enumerate}
\item $F\circ id \Rightarrow F$ is the identify
\item The induced natural isomorphisms $Fhk \Rightarrow Fk$ satisfy associativity, in the sense that for all $l,h,k \in K$,
\begin{tikzcd}
Fk(hl)   \dar[Rightarrow] \rar[Rightarrow] & Fhl  \dar[Rightarrow] \\
F(kh)l \rar[Rightarrow] & Fl
\end{tikzcd} commutes.
\end{enumerate}
\end{Def}
Here the left vertical arrow is the identity, the functors $k(hl)$ and $(kh)l$ are equal as functors on $P$.

\textbf{Note:} In addition to the natural isomorphism $Fk  \Rightarrow F$, we have a natural isomorphism $F \Rightarrow Fk$, obtained from $k^{-1}$, which we will use.

To put this more explicitly, we have an action of $K$ on $P$, and for a morphism $X \to Y$ in $P$, we have a diagram:
\begin{tikzcd}
F(X) \dar  & F(kX) \dar \lar{k^*} \\
F(Y)  & F(kY) \lar{k^*}
\end{tikzcd}

\begin{lemma}%again, characterize the uniquness of this
A $K$-Duflot module $L$ has a $P_W$-linear right $K$ action, uniquely characterized by: for all $k \in K$, the diagram:
\begin{tikzcd}
F(X) \dar{k^*} \rar & L \dar{k^*} \\
F(k^{-1}X) \rar & L
\end{tikzcd} commutes.
\end{lemma}
\begin{proof}
We show how this data defines an action on $L$, and then uniqueness follows from the universal property of  colimits.

For $k \in K$, we define a map $k^*:L \to L$ as follows. For all $X \in P$, we need a map $F(X) \to L$ commuting with the maps $F(X) \to F(Y)$ defining our colimit. For $X \in P$, we define $F(X) \to L$ by the map $F(X) \to[k^*] F(k^{-1}X) \to L$, where the second map is the natural map $F(k^{-1}X) \to L$. That this defines a system of compatible maps follows from the commutativty of the diagram:
\begin{tikzcd}
F(X) \dar{k*} \rar & F(Y) \dar{k^{*}}\\
F(k^{-1}X) \rar & F(k^{-1}Y)
\end{tikzcd}. That it defines a group action follows from the conditions on the natural transformations.
 % namely that multiplication by the identity acts as the identify and that the composition $F(hkx) \to [h^*] F(kx) \to [k^*] F(x)$ is the composition $F( (hk)x) \to[hk^*] F(x)$, i.e. that $(hk)^*=k^*h^*$.
\end{proof}

\begin{prop}
The action of $K$ on $L$ induces an action of $K$ on the associated graded: on each summand the map $k^*:\gr[j] F(kX) \to \gr[j] F(X)$ is the isomorphism induced by the isomorphism $F(kX)\to F(X)$, and the image of $\gr[j] F(kX)$ in the other summands is zero.
\end{prop}
\begin{proof}
This follows because the action by $k$  on $F(kX)$ fits into the diagram
\begin{tikzcd}
F(X) \rar & L  \\
F(kX) \uar{k^*} \rar & L \uar{k*}
\end{tikzcd}. This shows that the map is as claimed on these summands, and if $k^*(F(kX))$ intersects any other $F(Y)$ it is in higher filtration degree, so is in $\sum_{Z<Y} F(Z)$.
\end{proof}

\begin{Def}\label{thedefinitionofanequivariantmorphism}
For $K$-Duflot modules $L, N$, a morphism $\pi:(L,P,F) \to (N,Q,G)$ is \emph{equivariant} if:
\begin{enumerate}
\item $\pi_*: P \to Q$ is equivariant
\item For all $k \in K$, there is a commuting diagram
\begin{tikzcd}
G \pi \rar[Rightarrow] & F  \\
G k\pi \uar[Rightarrow] \rar[Rightarrow] & Fk \uar[Rightarrow]
\end{tikzcd}. In this diagram, the vertical natural transformation and the upper horizontal one are stipulated in the definitions of a $K$-Duflot module and a morphism, and the bottom one is induced from these.
\end{enumerate}
\end{Def}

\begin{lemma}\label{propertiesofequivariantmorphism}
The unique map of \ref{lemmaboutamorphisminducingmorphims} induced by an equivariant morphism is an equivariant map $\pi^*:N \to M$.
\end{lemma}
\begin{proof}
This follows from the commutativity of diagrams of the form:

\begin{tikzcd}
G(Y) \rar \dar{k^*} & \prod_{Y \in \pi^{-1} Y} F(X) \dar{k^*} \\
G(k^{-1}Y) \rar & \prod_{Z \in \pi^{-1} k^{-1}Y}F(Z)
\end{tikzcd}, because the lower right hand object is also $\prod_{X \in \pi^{-1}Y} F(k^{-1}X)$.
\end{proof}

\begin{Def}
A covering map of posets $P \to Q$ is a \emph{principal $K$-bundle} if $K=\mathrm{Aut}(P \to Q)$ acts transitively on all fibers, or equivalently if $P \to Q$ fits into a triangle
\begin{tikzcd}
P \dar \drar & \\
P/K \rar & Q
\end{tikzcd} where the bottom map is an isomorphism.
\end{Def}
By $\mathrm{Aut}(P \to Q)$ we mean the group of poset isomorphisms $P \to P$ fitting into triangles: 
$\begin{tikzcd} P \rar \dar \dar & P \dlar \\
Q
\end{tikzcd}$.

\begin{Def}
If $\pi:(L,P,F) \to (N,Q,G)$ is a morphism where all the maps $Q(\pi(X)) \to F(X)$ are isomorphims, $\pi:P \to Q$ is a principal $K$-bundle, $L$ is $K$-Duflot with respect to the $K$ action on $P$, $N$ is $K$-equivariant with trivial action, and $\pi$ is an equivariant map with respect to these actions, then we say that $\pi:(L,P,F) \to (N,Q,G)$ is a $K$-bundle.
\end{Def}

\begin{lemma}
If $\pi:(L,P,F) \to (N,Q,G)$ is a $K$-bundle, then $\pi^*:N \to L^K$ is an injection.
\end{lemma}
\begin{proof}
First, because the map is equivariant and because the action on $N$ is trivial, $\pi^*$ must map into the invariants. Then to check injectivity it is enough to check  on associated gradeds, where the result is clear because the maps $Q(\pi(X)) \to F(X)$ are all isomorphisms. 
\end{proof}

\begin{prop}
If $\pi:(L,P,F) \to (N,Q,G)$ is a $K$-bundle, then $DN^* \to (DL^*)^K$ is an isomorphism, and each $DL^i$ is a sum of free $K$-modules.
\end{prop}
\begin{proof}
That each $DL^i$ is a sum of free modules comes from the computation of the action of $K$ on the subquotients of the filtration, because the action on the poset $P$ is free, and because the summands of $DL^i$ are indexed on $P_i$. %say this in some nicer way.

For the map $DN^* \to (DL^*)^K$, this follows because on associated gradeds, $F_j/F_{j+1}N \to F_j/F_{j+1}L^K$ is an isomorphism. This is so because over one summand of $F_j/F_{j+1}N$, the map is the diagonal map $\gr[j]G(X) \to \sum_{Y \in \pi^{-1}X}( \gr[j]F(Y)) \cong \ind_1^K \gr[j]F(X)$.

\end{proof}
\begin{thm}\label{spectralsequenceforKduflotmodules}
If $\pi:(L,P,F) \to (N,Q,G)$ is a $K$-bundle, then there is a spectral sequence $H^p(K,\HH^q(L)) \Rightarrow \HH^{p+q}N$.
\end{thm}
\begin{proof}
This is the one of the two hypercohomology spectral sequences associated to the complex of $K$ modules $DL^*$. One spectral sequence has the $E_2$ term listed in the statement of the theorem, and the other spectral sequence has $E_2$-term $H^{p+q}( H^p(K,DL^*))$. But $H^p(K,DL^*)$ is zero above degree $1$ and is $DN^*$ in degree $0$, so the hypercohomology is $\HH^{*} N$ as claimed.
\end{proof}

\section{Connections with equivariant cohomology}\label{chapteronequivariantcohomology}
In this chapter we show that the algebraic structures studied in the previous chapter are present in the $S$-equivariant cohomology rings of smooth manifolds, and we apply the theorems of the previous chapter to deduce structural results in $S$-equivariant cohomology. Recall that all cohomology is with $\FF_p$ coefficients.
\subsection{$H_S^*M$ has a Duflot free rank filtration}\label{propertiesofHSM}

Now we show that this theory isn't vacuous: $H_S^*M$ has a Duflot algebra structure, where $S$ is a $p$-torus and $M$ is a smooth $S$-manifold. For applications to group cohomology, the most interesting $S$-manifold is  $M=G \backslash U(V)$,  where we have a faithful representation $G \to U(V)$ and $S$ is the maximal diagonal $p$-torus of $U(V)$ acting on the right of $G \backslash U(V)$. This manifold has two useful geometric properties:
\begin{enumerate}
\item For each subgroup $A \subset S$, each component $Y$ of $G\backslash U(V)^A$ is $S$-invariant. 
\item For each subgroup $A \subset S$ and component $Y$ of $G \backslash U(V)^A$, the open submanifold of $Y$ consisting of those points that have isotropy exactly equal to $A$ is connected.
\end{enumerate}
\begin{Def}We call $S$-manifolds satisfying these two properties \emph{$S$-connected}.
\end{Def}
 There is a straightforward modification of the definition of a free rank filtration that applies to smooth $S$-manifolds that aren't $S$-connected, but it is more complicated to state so we don't do so here.

In this section, fix a $p$-torus $S$ and a smooth $S$-connected manifold $M$.

\begin{lemma}
For $A<S$ and $Y$ a component of $M^A$, we have that $Y$ is also $S$-connected.
\end{lemma}
\begin{proof}
This follows from the fact that a component of $Y^B$ is also a component of $M^B$. %check this!
\end{proof}

\begin{Def}
Let $M_i^j=\{ x \in M: i \le \rank S_x \le j\}$.
\end{Def}

\begin{prop}[Duflot \cite{duflotfiltration} Theorem 1]\label{theoremofduflotgivingfiltration}
Define a filtration of $H_S^*M$-modules on $H_S^*M$ by $F_j= \ker H_S^*M \to H_S^*M_0^{j-1}$. Then: $$F_j/F_{j+1}= \bigoplus_{ \{ V \subset S: \rank V=j\} }\left( \bigoplus_{[Y] \in \pi_0(M^V)} \Sigma^{d_Y} H_S^*Y_j^j\right),$$ where $d_Y$ is the codimension of $Y$ in $M$. 

An alternative way to write this is as $F_j/F_{j+1}= \bigoplus_{[M'] \in \pi_0 M_j^j} \Sigma^{d_Y} H_S^*M'$, where $Y$ is the component of the fixed points of a rank $j$ $p$-torus such that $Y$ contains $M'$.
\end{prop}
\begin{proof}
Duflot's original proof shows this, although it is not stated there in this generality. Duflot states the result for $p\not=2$ and only gives a filtration as vector spaces, however her proof works verbatim to give the result here. A  treatment is found in chapter four of \cite{totaro}.
\end{proof}

\begin{thm}\label{duflotsfreerankfiltration}
We have that $H_S^*M$ has a free rank filtration.
\end{thm}
\begin{proof}
We use the notation as in Proposition \ref{theoremofduflotgivingfiltration}.

This is essentially just a computation of each $H_S^*M'$. First, the points of $M'$ have isotropy of rank exactly equal to $j$, so there is a unique rank $j$ $p$-torus $V$ with $M'^V=M'$. So, $V$ acts trivially on $M'$, so $H_S^*M'=H_V^* \otimes H_{S/V}^*M'$. But $S/V$ acts freely on $M'$, so this is $H_V^* \otimes H^* M'/(S/V)$. But $H_V^*=S(V^*)$ when $p=2$ and is $S( \beta(V^*) ) \otimes \Lambda (V^*)$ when $p$ is odd.

Finally, we see that we have the diagram of $S$-spaces:
\begin{tikzcd}
& pt \\
M\urar & Y_j^j \lar \uar \\
& S/V \ular \uar
\end{tikzcd}.

Applying $H_S^*$ gives:
\begin{tikzcd}
& P_S \dar \\
& H_S^* \dar  \dlar \\
H_S^*M \rar \drar & H_S^*Y_j^j \dar \\
& H_S^* S/V \dar \\
& P_V
\end{tikzcd}.%%modify for p-odd

In this diagram, when $p=2$ the top and bottom vertical arrows are just the identity.

This shows that each $\Sigma^{d_Y} H_S^*Y_j^j$ is $j$-free, so we are done.
\end{proof}

To show how to refine this free rank filtration to get a Duflot free rank filtration, we first observe that all the associated primes of $H_S^*M$ are toral in the sense of definition \ref{algebraicdefinitionoftoral}, and also toral in the sense of being pulled back from the map $H_S^*M \to H_S^*S/V/ \sqrt{0}$ induced by a map $S/V \to M$.
{}
\begin{lemma}\label{HSMtoral}
Every associated prime of $H_S^*M$ is toral in the sense of \ref{algebraicdefinitionoftoral}, and toral primes are exactly the primes that come from restricting to a $p$-torus.
\end{lemma}
\begin{proof}
That the associated primes are toral in the sense of \ref{algebraicdefinitionoftoral} follows immediately from \ref{algebraictoraltheorem} once we know that $H_S^*M$ has a free rank filtration.

To see that this definition of toral coincides with restriction to a $p$-torus, in the diagram above that shows that the Duflot filtration is a free rank filtration, observe that the map $H_S^*M \to P_V$ is induced by $S/V \to M$.
\end{proof}

\begin{lemma}\label{bundleisorientable}
For $Y$ a component of $M^A$ and for $p\not= 2$, the normal bundle of $Y$ in $M$ is $S$-equivariantly orientable.
\end{lemma}
\begin{proof}
First we show that the normal bundle is orientable. Denote the total space of the normal bundle by $N$, and consider the bundle $N \to Y$ as an $A$-vector bundle, i.e. restrict the action to $A$. Now $A$ acts trivially on $Y$, and because $Y$ is a component of $M^A$, no trivial representations of $A$ appear in the isotypical decomposition of $N \to Y$. Therefore, since every irreducible real representation of $A$ has a complex structure, $N \to Y$ is in fact a complex vector bundle, and therefore orientable.

Now, in order to show that the normal bundle is $S$-equivariantly orientable, we need to show that the $S$-action preserves a given orientation. For this, consider the bundle of orientations $O \to Y$ of the bundle $N \to Y$; by previous discussion this is a trivial $\ZZ/2$-bundle. Now the action of $S$ on $O \to Y$ gives a map $S \to \ZZ/2$ which is trivial if and only if $S$ preserves the orientation. But $S$ is a $p$-group, so the result follows.
\end{proof}

\begin{prop}\label{gysinisinjective}
Let $Y$ be a component of $M^A$. Then $H_S^*Y$ is fixed in $H_S^*M$.
\end{prop}
\begin{proof}
To see that $H_S^*Y$ is fixed in $H_S^*M$, we need to show that the composition of pushfoward and restriction $\Sigma^{d_Y} H_S^*Y \to H_S^*M \to H_S^*Y$ is injective, and to show that this composition is injective we just need to show that the equivariant Euler class of the normal bundle of $Y$ is a non-zerodivisor. We are using \ref{bundleisorientable} to guarantee the existence of a pushforward. So, suppose that the Euler class $e$ is a zero divisor. Then it is contained in one of the associated primes of $H_S^*Y$. These associated primes are all toral, so $e$ must restrict to $0$ under a map $S/B \to Y$. But $Y$ is fixed by $A$, so the map $S/B \to Y$ fits in the diagram $S/A \to S/B \to Y$, so $e$ must restrict to zero on $S/A$.

However,  when the normal bundle is restricted to $S/A$, we have an $A$ representation with no trivial summands, so its Euler class is nonzero.
%%why is it orientable?
 \end{proof}
 
 \begin{prop}\label{compatibility}
For $Y$ a component of $M^A$, the map $i_*: \Sigma^{d_Y} H_S^*Y \to H_S^*M$ induces a map from the Duflot filtration of $Y$ to the Duflot filtration of $M$, which induces an inclusion from a suspension of the Duflot complex of $Y$ to the Duflot complex of $M$.
\end{prop}

\begin{proof}
To see that we have the induced map on Duflot filtrations, we use the naturality and functoriality of the pushforward map. We have the following diagram, where each square is a pullback square:

\begin{tikzcd}
Y^i_i \rar{i_1} \dar{j_1}& Y_0^i \dar{j_2} & \dar \lar Y_0^{i-1}  \dar \\
M_i^i \rar{i_2} & M_0^i  & \lar M_0^{i-1}
\end{tikzcd}
 This induces the following commutative diagram, which is one piece of the Duflot filtration. In the following diagram, the two left horizontal arrows and the three slanted arrows are pushforwards, and the other maps are restrictions. We omit the suspensions for clarity.
 
$$\begin{tikzcd}
& H_S^*M_i^i  \rar & H_S^*M_0^i  \dar\\
H_S^*Y_i^i \arrow{ur} \rar & H_S^*Y_0^i \dar \arrow{ur} & H_S^*M_0^{i-1}  \\
& H_S^*Y_0^{i-1} \arrow{ur} &
\end{tikzcd}$$

Note that $H_S^*Y_i^i \to H_S^* M_i^i$ is the inclusion of a summand, since the map $Y_i^i \to M_i^i$ is the inclusion of a component. After applying local cohomology we get the following ladder:

\begin{tikzcd}
\cdots \rar \dar & \HH^i(H_S^*Y_i^i) \rar{d^i_Y} \dar & \HH^{i+1}H_S^*(Y_{i+1}^{i+1} ) \dar  \rar & \cdots  \dar\\
\cdots  \rar & \HH^i(H_S^*M_i^i) \rar{d^i_M} & \HH^{i+1}H_S^*(M_{i+1}^{i+1} ) \rar & \cdots 
\end{tikzcd}

Here the vertical arrows are inclusions of summands, since local cohomology is an additive functor.

\end{proof}

\begin{Def}\label{definitionoftheposetfixm}
Let $\Fix(M)$ be the weakly coranked poset whose elements are components $Y$ of $M^V$, where $V$ is some subtorus of $S$, and where $Y_{\rank V}^{\rank V}$ is non empty, and where the morphisms are given by inclusions. The rank of $Y$ is the rank of the largest $p$-torus that fixes it.
\end{Def}

There is a functor $F:\Fix(M) \to H_S^*M\text{-mod}$ sending $Y$  to $\Sigma^{d_Y} H_S^*Y$. We now show that this refines the filtration from \ref{duflotsfreerankfiltration}.
\begin{lemma}
The filtration on $H_S^*M$ by $\Fix(M)$ refines the filtration we have already defined, i.e. $\sum_{ Y \in \Fix(M)_{\ge s} } F(Y)=F_s H_S^*M$.
\end{lemma}
\begin{proof}
Let $J_s=\sum_{Y \in \Fix(M)_{\ge s} }F(Y)$. First, we observe that $J_s \subset F_s$: this is because for each $Y$ $i_*:H_S^*Y \to H_S^*M$ respects the Duflot filtration, and because $s$ is the lowest possibly nonzero term in the Duflot complex for $Y$, where $Y$ is a connected component in the fixed point set for a rank $s$ $p$-torus. 

The other inclusion is by downward induction on $s$. First, when $s= \dim H_S^*M$, then $F_s H_S^*M=F_s/F_{s+1}H_S^*M$, which by definition is $J_s$.

So, suppose $F_{i+1}=J_{i+1}$, and consider $x \in F_{i}$. By our description of $F_{i}/F_{i+1}$ from the first part of Proposition \ref{theoremofduflotgivingfiltration}, there is some $y \in J_{i+1}$ so that $x$ and $y$ are equal modulo $F_{i+1}$, so we are done by induction.
\end{proof}

\begin{lemma}\label{thefiltrationisgood}
The filtration on $H_S^*M$ induced by $\Fix(M)$ is good.
\end{lemma}
\begin{proof}
We need to show that the map $\bigoplus_{Y \in \Fix(M)_j } \gr[j] F(Y) \to F_j/F_{j+1}H_S^*M$ is an isomorphism. This follows from \ref{compatibility}: $F_j/F_{j+1}$ has a direct sum decomposition indexed on $\Fix(M)_j$, and the map $F(Y) \to H_S^*M$ for $Y \in \Fix(M)_j$ induces a map $\gr[j] F(Y) \to F_j/F_{j+1}H_S^*M$ which is the inclusion of the direct summand corresponding to $Y$.
\end{proof}

\begin{prop}\label{duflotfiltrationonFiHSM}
The poset $\Fix(M)_{\ge i}$ puts a Duflot module structure on $F_iH_S^*M$.
\end{prop}
\begin{proof}
We have seen in \ref{thefiltrationisgood} that the filtration by $\Fix(M)$ is a stratified free rank filtration. 

To see that it is topological, we use \ref{gysinisinjective} and \ref{compatibility}.

To see that it is minimal, we need to check that for each $Y \in \Fix(M)$ each $H_S^*M \to H_S^*Y$ has a distinct kernel. For $Y \in M^A$, we have the diagram:
\begin{tikzcd}
H_S^* \dar \drar \arrow{drr} &&  \\
H_S^*M \rar & H_S^*Y \rar & H_S^*S/A
\end{tikzcd}
The maps $H_S^* \to H_S^*S/A=H^*_A$ have distinct kernel for distinct $A$, so we only need to show that if $Y,Y'$ are different components of $M^A$, then the restrictions to $Y$ and to $Y'$ have different kernels. However, this is immediate because $Y$ and $Y'$ are disjoints, so the compostion $\Sigma^{d_Y} H_S^*Y \to H_S^* M \to H_S^*Y'$ is zero, while $\Sigma^{d_Y} H_S^* Y \to H_S^*M \to H_S^*Y$ is nonzero.

Finally, to see that it is Duflot, we use the computation that if $Y$ is a component of $M^A$, then $H_S^*Y \cong H_A^* \otimes H^*_{S/A}Y$.

\end{proof}

\subsection{Properties of $H_S^*M$}\label{algebraicpropertiesofHSM}
Here we collect various algebraic properties of $H_S^*M$ that follow from $H_S^*M$ having a Duflot algebra structure. 
% We will use these in \ref{theorems}.

As in the previous section, here $M$ is a smooth $S$-connected manifold.
\begin{cor}[Duflot \cite{duflotdepth} Theorem 1, Quillen \cite{quillen} Theorem 7.7]\label{depthanddimensionforHSM}
The depth of $H_S^*M$ is no less than the largest rank of a subtorus of $S$ that acts trivially on $M$, and the dimension is equal to the largest rank of a subtorus of $S$ that acts with fixed points on $M$.
\end{cor}
\begin{proof}
This follows from \ref{depthanddimoftruncations}. If we let $d$ be the largest rank of a subtorus of $S$ that acts trivially on $M$ and $r$ the largest rank of a subtorus that acts with fixed points on $M$, then $F_d H_S^*M=H_S^*M$, and $F_{r} H_S^*M$ is the smallest nonzero level of the filtration. Note that \ref{depthanddimoftruncations} only gives that the dimension of $H_S^*M$ is less than or equal to $r$, but by studying the dual of the Duflot chain complex (see Propostion \ref{matlisdualcomplex}) for $H_S^*M$ we conclude that $\HH^r{ H_S^*M}$ is nonzero.
\end{proof}

\begin{cor}[Duflot \cite{duflotassociatedprimes} Theorem 3.1]
Each associated prime of $H_S^*M$ is toral.
\end{cor}
\begin{proof}
This is just Lemma \ref{HSMtoral}.
\end{proof}
\textbf{Note}: The proof we have given here is essentially the same as Duflot's proof.

\begin{cor}
There is a cochain complex $DM$ with $H^i(DM)=\HH^i (H_S^*M)$, and $DM^i=\oplus_{Y \in \Fix(M)_i}\Sigma^{d_Y} \HH^i( H_S^*Y_i^i)$.
\end{cor}
\begin{proof}
This follows from \ref{duflotcomplex}.
\end{proof}

\begin{remark}
This result should be compared with the compuation of the Atiyah-Bredon complex in \cite{alldaypuppefranzsyzgies} Theorem 1.2. There, Allday, Franz, and Puppe study a chain complex occuring in the rational torus equivariant cohomology of a manifold, the complex is analagous to the dual of the Duflot chain complex. In unpublished work the authors extend some of their results to the setting of $p$-torus equivariant cohomology.
\end{remark}

\begin{cor}[Symonds \cite{symondsregularity}  Proposition 4.1]\label{regularityforHSM}
The regularity of $H_S^*M$ is less than or equal to the dimension of $M$.
\end{cor}
\begin{proof}
This follows from \ref{regularitytheorem}. Each $j$-free summand appearing in the free rank filtration has the form $\Sigma^{d_Y} H_V \otimes H^*(Y_j^j/S)$, and $d_Y + \dim Y_j^j/S= \dim M$, so the result is as claimed. 

The distinction between $p=2$ and $p \not=2$ appearing in \ref{regularitytheorem} does not change the result because of the exterior terms in $H_V^*$ cancel out the extra $j$ term.
\end{proof}

\begin{cor}[Symonds \cite{symondsregularity}]\label{strongregularityforHSM}
If $a_i(H_S^*M)=-i+\dim M$, then there is a maximal element of $\Fix(M)$ of rank $i$.
\end{cor}
\begin{proof}
Consider the Duflot chain complex for $H_S^*M$. The $i^{th}$ degree term of this chain complex is a sum of modules of the form $\HH^i( \Sigma^{d_Y} H^*_V) \otimes H^*(Y_i^i/S)$, where $V$ is a rank $i$ p-torus. The top nonzero degree of this will be $d_Y + -i + H^{\text{top}}H^*(Y_i^i/S)$. We have that $Y_i^i/S$ is a manifold of dimension $\dim M - d_Y$, so $a_i(H_S^*M)=-i+\dim M$ only if there is some $Y$ so that $H^{\dim Y} Y_i^i/S$ is nonzero, which happens if and only if $Y_i^i/S$ is compact, and consequently if and only if $Y_i^i$ is  compact. 

However, $Y_i^i=Y -\cup_{W \in \Fix(M)_{> Y}} W$ (recall that every point of $Y$ has isotropy of rank greater than or equal to $i$), so $Y_i^i$ is compact if and only if $Y_i^i=Y$ and $Y$ is maximal in $\Fix(M)$.
\end{proof}
\begin{remark}
Symonds shows this without explicitly stating it in his proof of some special cases of the strong regularity theorem in \cite{symondsregularity} Propostion 9.1. 
\end{remark}

\begin{cor}\label{assoicatedprimesforHSM}
Take $(Y,A) \in \Fix(M)$, and let $\p$ be the kernel of $H_S^*M \to H_S^*Y \to H_S^* S/A$. The following are equivalent.
\begin{enumerate}
\item $\p \in \ass_{H_S^*M} H_S^*M$.
\item $\p \in \ass_{H_S^*Y} H_S^*Y$.
\item The depth of $H_S^*Y$ is $\rank Y$.
\end{enumerate}
\end{cor}
\begin{proof}
This follows from \ref{depthofcentralizers}.
\end{proof}

\begin{cor}[Quillen \cite{quillen} Propostion 11.2]\label{minimalprimesofHSM}
The minimal primes of $H_S^*M$ are in bijection with the minimal elements of $\Fix(M)$, i.e. the pairs $A,Y$ where $A$ is a $p$-torus of $Y$ and $Y$ is a component of $Y^A$, and $S/A$ acts freely on $Y$. The primes are obtained via the kernel of $H_S^*M \to H_S^*Y \to H_S^*S/A$.
\end{cor}
\begin{proof}
It is clear by our description of the associated primes that if the primes arising in this manner are associated, then they are minimal among the associated primes and therefore minimal primes. That there are  distinct primes for each maximal element was noted in the proof of \ref{duflotfiltrationonFiHSM}. 

To see that these primes are associated we use \ref{assoicatedprimesforHSM}. Since $S/A$ acts freely on $H_S^*Y$, $H_S^*Y$ is isomorphic to $H^*_A \otimes H^*Y/S$, so the Duflot bound is sharp for $H_S^*Y$ and the result follows.
\end{proof}
This proof is different than Quillen's proof. In \cite{quillen} he obtains this as a consequence of the F-isomorphism theorem.

The following is the $S$-equivariant cohomology version of a detection result of Carlson \cite{carlsondepthconjecture}. We will use this later to derive Carlson's detection result.
\begin{cor}\label{detectionforHSM}
If $H_S^*M$ has depth $d$, then $H_S^*M$ is detected by restricting to the $H_S^*Y$ for $Y \in Fix(M)_d$.
\end{cor}
\begin{proof}
This follows from \ref{detection}.
\end{proof}

Recall Carlson's conjecture:
\begin{con}[Carlson]
If $H_G^*$ has depth $d$, then $H_G^*$ has a $d$-dimensional associated prime.
\end{con}

We will use the following to prove Carlson's conjecture in the special case where $G$ is a compact Lie group with the Duflot bound sharp.

\begin{thm}\label{carlsonconjectureforHSM}
If the Duflot bound for $H_S^*M$ is sharp, then there is associated prime of dimension the depth of $G$ given by $H_S^*M \to H_S^*S/A$, where $A$ is the maximal $p$-torus that acts trivially on $M$.
\end{thm}
\begin{proof}
This follows from the third part of \ref{assoicatedprimesforHSM}: $M^A=M$, so $M$ is a minimal element of $\Fix(M)$ of rank $\dim A$, so if $\depth H_S^*M=\dim A$, then $H_S^*M \to H_S^* S/A$ represents an associated prime.

%Let $A$ be the largest $p$-torus that acts trivially on $M$. 
%Then $H_S^*M \cong H_A^* \otimes H_{S/A}^*M$. 
%But if the Duflot bound is sharp, then the depth of $H_S^*M$ is $\dim A$. Then because depth is additive under tensor products (this follows for example from the K\"unneth theorem for local cohomology), $H_{S/A}^*M$ has depth $0$. 
%
%Therefore, there is some element $x \in H_{S/A}^*M$ that is anihillated by the irrelevant ideal of $H_{S/A}^*M$. Let $y$ be an element of $H^*_A$ such that $\ann(y)$ is the unique minimal prime of $H^*A$, the kernel of $H^*A \to H^*A/\sqrt{0}$. When $p=2$ $y=1$ is such an element, and for $p$ odd the product of all the exterior classes in $H^*A$ is such an element. Then $\ann(1 \otimes x)=\ker(H_S^*M \to H_S^* S/A$, so we are done.
\end{proof}
% \begin{remark}
% It is worth  stating that this proof does not essentially use the machinery of a Duflot filtration, and this can be converted to an elementary proof of Carlson's conjecture when the Duflot bound for depth is sharp.
% \end{remark}

\subsection{Exchange between $S$ and $G$ equivariant cohomology}\label{exchange}

The $S$-manifolds we are primarily interested in are those of the form $G \backslash U(V)$, where $G$ is a compact Lie group and $G \to U(V)$ is faithful. Recall that by convention $U(V)$ respects direct sum decompositions of $V$. In other words, if we have a direct sum decomposition  $V=V_1 \oplus \dots \oplus V_n$ of a faithful unitary representation $V$ of $G$, then $U(V)=\prod_{i=1}^n U(V_i)$.

The space $U(V)$ has a maximal $p$-torus $S$ of diagonal matrices of order dividing $p$. We take  $G$ to act on the left of $U(V)$ and $S$ to act on the right.
 % so we write $G \backslash U(V)$ instead of $U(V)/G$. 
 Our first main goal is to  show is that $G \backslash U(V)$ is an $S$-connected manifold. This will show that $H_S^*G \backslash U(V)$ has a Duflot algebra structure, and then we will see that $H_G^* U(V) /S$ does as well.

We need to explain the relationships between three closely related manifolds. First we have $U(V)$. The left action of $G$ and the right action of $S$ gives an action of $G \times S$ on $U(V)$ (if this is to be a left action, we must modify the right action of $S$ so that $s$ acts by $s^{-1}$, and if we want this to be a right action we similarly modify the $G$ action). We also have the left action of $G$ on $U(V)/S$, and the right action of $S$ on $G \backslash U(V)$.

Note that there are maps $G \backslash U(V) \from U(V) \to U(V)/S$. If $A < S$ and $Gu \in G \backslash U(V) ^A$, then $^uA< G$, and $uS \in (U(V)/S)^{(^uA)}$. Here we mean the action of $^uA$ as a subgroup of $G$, so it is acting on the right.

 Let $D < {}^uA \times A < G \times S$ be the image of $A$ under $a \mapsto (uau^{-1},a)$. We give $D$ the action on $U(V)$ that is restricted from $G \times S$ acting on $U(V)$ on the left, so $(uau^{-1},a) \cdot v= uau^{-1}v a^{-1}$.

Parts of the following lemma are used in \cite{symondsregularity}. These results could also be derived from
\cite{duflotassociatedprimes}, but for completeness we include a proof here.
\begin{lemma}\label{lemmaaboutexchange}
\begin{enumerate}
\item For $A < S$ and $x=Gu \in (G \backslash U(V))^A$, the connected component $Y$ of $x$ in $(G \backslash U(V))^A$ is isomorphic as $S$-spaces to $C_{G^u} A \backslash C_{U(V)} A$. 
\item Moreover, the connected component $X$ of $xS \in U(V)/S^{^uA}$ is $C_G (^uA)$-invariant and isomorphic as $C_G(^uA)$-spaces to $C_{U(V)} 
^uA / ^uS$, and the $G$ orbit of the connected component is $G \times_{C_G (^uA) } C_{U(V)} (^u A )/ ^uS$.
\item Given $Gu \in (G \backslash U(V))^A$, the connected component $Z$ of $u \in U(V)^D$ (where the $D$ is as in the discussion preceding the lemma) is the left $C_{U(V)} (^uA)$ orbit of $u$, and the right $C_{U(V)} A$-orbit of $u$ (in fact, $U(V)^D$ has one connected component). Under the maps $G \backslash U(V) \from U(V) \to U(V)/S$, $Z$ maps onto to $Y$ and onto $X$. We have that $Z$ is invariant under the right $S$ action, but not necessarily under the left $G$ action. However, the $G$-orbit of $Z$  is $G \times_{C_{G} (^uA)} C_{U(V)} A$ and it maps onto the $G$-orbit of $X$, and the $G$ orbit of $Z$ also maps onto $Y$.
\end{enumerate}
\end{lemma}
\begin{proof}
\begin{enumerate}
\item
First, observe that if $Gu$ is fixed by $A$, then $^u A < G$, and $A < G^u$. Also, it is immediate that $C_{U(V)}A$ acts on the right on $(G \backslash U(V))^A$, and because $C_{U(V)}A$ is a product of unitary groups it is connected, and consquently the $C_{U(V)}A$-orbit of $x$ is connected.

We will first show that $Y$ is the $C_{U(V)}A$-orbit of $x$. It is enough to show that there are only finitely many orbits of the $C_{U(V)}A$ action on $(G \backslash U)^A$, from this it follows that each orbit is a component of $(G \backslash U)^A$. To see that there are only finitely many $C_{U(V)}A$-orbits is it enough to show that there are only finitely many $N_{U(V)}A$-orbits on $(G \backslash U(V))^A$, because $C_{U(V)}A$ has finite index in $N_{U(V)}A$.

If $Gu, Gv \in (G \backslash U)^A$ have the property that $^uA$ and $^vA$ are conjugate $p$-tori in $G$, so ${}^{gv}A={}^uA$ for some $g \in G$, then $u^{-1}gv \in N_{U(V)}A$, and $Gu$ and $Gv$ are in the same $N_{U(V)}A$-orbit. So, since $G$ has only finitely many conjugacy classes of $p$-tori, there are only finitely many $N_{U(V)} A$-orbits, and we have shown that $Y$ is the $C_{U(V)}A$-orbit of $x$.

To see that $Y$ is $C_{G^u} A \backslash C_{U(V)}A$, observe that the stabilizer of $Gu$ under the $C_{U(V)}A$ action is $C_{G^u}A$.
\item We can follow the same strategy as above to identify $X$. First note that $C_{U(V)} (^uA)$ acts on $U(V)/S^{^u A}$, then observe that the $C_{U(V)} (^uA)$-orbits are each a component of $U(V)/S^{^u A}$, and then that the stabilizer is as claimed. Once we have identified $X$ this also shows that $X$ is $C_G^{^u A}$-invariant, as $C_G (^uA) < C_{U(V)} (^u A)$.

To see that the $G$-orbit of $X$ is isomorphic to $G \times_{C_G (^uA) } X$ as $G$-spaces, note that $G$ acts on the components of the $G$-orbit of $X$, and there is an obvious equivariant homeomorphism from $G \times_{\stab X} X$ to the $G$-orbit of $X$. We have already observed that $C_G  (^uA)$ is contained in $\stab X$. To see the other containment, suppose that $gcuS=c'uS$, where $c,c' \in C_{U(V)} (^u A)$ and $g \in G$. Then $c'^{-1}gc \in ^uS < C_{U(V)} (^uA)$, so $g$ must centralize $^u A$.

\item If we can show that connected component of $u \in U(V)^D$ is as claimed, then the other claims follow. This is just a computation, it is easy to check that the left action of $C_{U(V)} (^u A)$ and the right action of $C_{U(V)}A$ preserve $Z$, and if $y \in U(V)^D$, then we can write $y=cu=uc'$, where $c \in C_{U(V)} (^uA)$ and $c' \in C_{U(V)}A$.
\end{enumerate}

\end{proof}

This is very nice, because if we start with a right $S$ space of the form $G \backslash U(V)$, then each component of $(G \backslash U(V))^A$ has the same form for some conjugate subgroup of $G$: the components are of the form $C_{G^u}A \backslash U(V')$, where $V'$ is a direct sum of representations of $C_{G^u}A$, giving a faithful representation of $C_{G^u}A$.

Similarly, if we start with the right $G$-space $U(V)/S$, then each component of $U(V)/S^B$ has the same form for some subgroup of $G$: the components are of the form $U(V') ^uS$.

We'll explore the implications of this for group cohomology in the next section. For now, we show that $G \backslash U(V)$ is $S$-connected. 

\begin{prop}
Let $M= G \backslash U(V)$. Then $M$ is $S$-connected.
\end{prop}
\begin{proof}
That $S$ preserves each component of $M^A$ is clear from our description of each component of $M^A$. So, we need to show that, for $Y$ a component of $M^A$, if $Y_{\rank A}^{\rank A}$ is nonempty  it is connected.

Via the above correspondence between connected components of the fixed point sets of $U(V)/S$ and $G \backslash U(V)$,  it is enough to show that $Z_{\rank D}^{\rank D}$ is connected, where the notation is as in \ref{lemmaaboutexchange}.

For this, note that for each $E$ a sub $p$-torus of $D$, $U(V)^E$ is the total space of a torus bundle over a (product of) flag manifolds, and for $U(V)^E \subset U(V)^D$, we have the following diagram of these $T$-bundles.
$$\begin{tikzcd}
T \dar \rar & T\dar \\
U(V)^E \rar \dar & U(V)^D \dar \\
F'' \rar & F' ,
\end{tikzcd}$$
In the above diagram $F'' \to F'$ is a proper sub (product of) flag manifolds. But $F''$ therefore has even codimension in $F'$, so $U(V)^E$ has even codimension in $U(V)^D$. 

However, $Z_{\rank D}^{\rank D}= Z - (\cup_{E <D} Z^E)$. Each $Z^E$ therefore has even codimension in $Z$, so $Z_{\rank D}^{\rank D}$ is connected and we are done.

Instead of looking at these flag manifolds, we could get this same result by comparing the dimension of $C_{U(V)}A$ and $C_{U(V)}B$ for $A<B$, each is a product of unitary groups corresponding to the isotypical decomposition of the representation for the respective $p$-torus, and if $A<B$ the isotypical decomposition for $B$ refines that for $A$.
\end{proof}

So, this tells us that all the results from \ref{propertiesofHSM} and \ref{algebraicpropertiesofHSM} apply to $H_S^* G \backslash U(V)$.

Moreover, we now show that these results also apply to $H_G^*U(V)/S$.
%%gysin?
\begin{prop}\label{correspondence}
Under the maps $G \backslash U(V)\from U(V) \to U(V)/S$, and for $Y,Z,X$ components of the fixed point sets of $A, D$ and $^u A$ respectively, as above, we have this diagram:

$$
\begin{tikzcd}
H_S^* Y \dar \rar & H_{G \times S}^* GZ \dar   & H_G^* GX \dar \lar \\
H_S^* G \backslash U(V)  \dar \rar & H_{G \times S}^* U(V)   \dar & H_G^*U(V)/S \dar \lar  \\
 H_S^* Y  \rar  & H_{G \times S}^* GZ    & H_G^* GX \lar
\end{tikzcd}.$$

The notation $GZ$ and $GX$ denote the $G$-orbits of $Z$ and $X$. In this diagram, all the horizontal arrow are isomorphisms.
\end{prop}
\begin{proof}
The diagram of the theorem comes from the diagram of spaces:

$$\begin{tikzcd}
ES \times_S Th(Y)  & EG \times ES_{G \times S} Th(GZ)   \lar \rar & EG \times_G Th(GX)  \\
ES \times _S G \backslash U(V) \uar  & EG \times ES_{G \times S} U(V) \lar \rar  \uar & EG \times U(V)/S  \uar \\
ES \times_S Y  \uar& EG \times ES_{G \times S} GZ  \rar \lar  \uar  & EG \times_G GX \uar
\end{tikzcd}.$$

The horizontal arrows in the bottom two rows are homotopy equivalences by our description of $Y, GZ,$ and $GX$. For the top row, we first apply the equivariant Thom isomorphism to the cohomology of the top row, and then our description of $Y, GZ$ and $GX$ gives us the result.

\end{proof}
% \begin{remark} It is worth noting that in order for this proof to be correct, we must verify that in the case that $p$ is odd all the bundles in question are actually orientable, so that the claimed Thom isomorphisms exist. However, this follows in a straightforward way from the normal bundle for $Y$ in $G \backslash U(V)$ being $S$-equivariantly orientable.
% \end{remark}

\begin{Def}
Let $\Fix(U(V)/S)$ be the weakly coranked poset whose elements are $G$-orbits of connected components of $U(V)/S^A$, for $A$ a $p$-torus of $G$. 
\end{Def}
%Note that $\Fix(M)$ is isomorphic to the homotopy orbit category of $M$, which is  denoted as $\mathcal{A}_S M$. This category has maps homotopy classes of $S$-equivariant maps $S/A \to M$, and the morphisms are the homotopy commutative triangles. The poset $Fix(U/S)$ has the same objects as $\mathcal{A}_G U/S$, but it has fewer morphisms when $G$ is not abelian. 
%put this somewhere else

%The category $\mathcal{A}_G U/S$ is closely related to the category used in \cite{quillen}.
\begin{lemma}
Under the correspondence between connected components of $G \backslash U(V) ^A$ and connected components of $U(V)/S^{ ^u A}$, $\Fix(G \backslash U(V)) \iso \Fix( U(V)/S)$.
\end{lemma}
\begin{proof}
This is immediate from the discussion of these connected components above.
\end{proof}

Note that there is a functor $F:\Fix(U(V)/S) \to H_G^*U(V)/S$-mod where $X$ is mapped to the module $(\Sigma^{d_X}H_G^*( G \times_{\stab X} X )$, which via the pushforward is equipped with an inclusion into $H_G^*(U(V)/S) )$.
\begin{thm}
$H_G^*U(V)/S$ is a Duflot algebra, stratified by $\Fix(U(V)/S)$.
\end{thm}
\begin{proof}
This follows immediately from \ref{correspondence}: we know that $H_S^* G \backslash U(V)$ has a Duflot algebra structure stratified by $\Fix( G \backslash U)$, and this puts a Duflot algebra structure on $H_G^*U(V)/S$ via the isomorphisms in \ref{correspondence}.
\end{proof}

\section{Applications to the cohomology of $BG$}\label{sectiononthecohomologyofBG}
In this chapter we apply the Duflot algebra structure on $H_S^*U(n)/G$ to study $H^*BG$. In section \ref{theorems} we give structural results that apply to $H^*_G$ for any compact Lie group $G$, and in sections \ref{itrivialbundles} and \ref{wreathproducts} we show how certain extensions of finite groups lead to the structures studied in \ref{KDuflot}, from which we derive some local cohomology computations for the $p$-Sylow subgroups of $S_{p^n}$ and for $S_{p^n}$ itself.

\subsection{Structural results in group cohomology}\label{theorems}
Here we demonstrate how the results of this paper recover several classical results in group cohomology. All of the results listed in this section were previously known at least for finite groups, but except where indicated otherwise with  different proofs from our methods. There are several new results for compact Lie groups.

Throughout this section, $G$ is a compact Lie group and $V$ a faithful finite dimensional unitary representation of $G$.

Crucial to applying the techniques of this paper to get results in group cohomology is a theorem of Quillen \cite{quillen}  showing to what extent $H_G^*$ can be recovered from $H_G^*U(V)/S$.
\begin{prop}[Quillen \cite{quillen} Lemma 6.5]\label{faithfullyflat}
The map $H_G^* \to H_G^*U(V)/S$ is faithfully flat, and as $H_G^*$-modules $H_G^*U(V)/S$ is non-canonically isomorphic to $H_G^* \otimes H^*U(V)/S$.
\end{prop}

\begin{cor}
The map $H_G^* \to H_G^*U(V)/S$ induces a surjection from $\ass H_G^*U(V)/S$ to $\ass H_G^*$.
\end{cor}
This is a consequence of the map being faithfully flat.

\begin{thm}[Duflot \cite{duflotdepth} Theorem 1, Quillen \cite{quillen} Theorem 7.7]\label{duflotandquilleningroupcohomology}
The depth of $H_G^*$ is greater than or equal to the $p$-rank of the center of $G$, and the dimension of $G$ is equal to the maximal rank of a $p$-torus of $G$.
\end{thm}
The result on depth is due to Duflot, and the result on dimension is due to Quillen. Duflot's original proof does not use the Duflot filtration in the way that we do here, but it has the advantage that it explicitly constructs a regular sequence of length the lower bound for depth.

The result on dimension follows from the Quillen stratification theorem. Here we give a proof of this theorem using the Duflot filtration.
\begin{proof}
By \ref{faithfullyflat}, it is enough to show the analogous result for $H_G^*U(V)/S=H_S^*G \backslash U(V)$. However, we note that we can choose for $V$ a sum of representations for $G$ so that $Z(G)$ maps to $Z(U(V))$. With this choice of $V$, the result follow from \ref{depthanddimensionforHSM}.
\end{proof}
\begin{remark}\label{strengtheningofduflotbound}
Duflot's bound can be made  stronger in the case that $G$ is a finite group: the depth is in fact no less than the maximal rank of a central $p$-torus in a $p$-Sylow of $G$. This stronger statement follows immediately from the statement we have given, by using the fact that for $P$ a $p$-Sylow of $G$, $H_G^*$ is a summand of $H_P^*$ as $H_G^*$-modules: if $\HH^i H_G^*$ is nonzero, so is $\HH^i H_P^*$.
\end{remark}

Quillen also showed that the minimal primes of $H_G^*$ are given by restricting to a maximal (by inclusion) $p$-torus. We can also get this result using our techniques.
\begin{thm}[Quillen \cite{quillen} Propostion 11.2]
The minimal primes of $H_G^*$ are obtained by restricting to maximal $p$-tori.
\end{thm}
\begin{proof}
Because $H_G^* \to H_G^* U(V)/S$ is faithfully flat, all the minimal primes of $H_G^*U(V)/S$ pull back to minimal primes of $H_G^*$. But \ref{minimalprimesofHSM} describes the minimal primes of $H_S^*U/S$, which are given by restricting to minimal elements of $\Fix(U/S)$. These minimal elements correspond to a maximal $p$-torus $A$ of $G$ and a component of $U/S^A$, so it is clear that all minimal primes of $H_G^*$ have the claimed description. 

\end{proof}
\begin{remark}
Quillen in fact shows that there is a bijection between the conjugacy classes of maximal $p$-tori and the minimal primes of $H_G^*$. In order to conclude this stronger result from what we have done so far, it is only necessary to show that if $A$ and $B$ are non-conjugate $p$-tori in $G$ then $H_G^* \to H_G^* G/A/ \sqrt{0}$ and $H_G^* \to H_G^*G/B /\sqrt{0}$ have different kernels. Note that we already have shown that $H_G^* U(V)/S \to H_G^* G/A/ \sqrt{0}$ and $H_G^*U(V)/S \to H_G^*G/B /\sqrt{0}$ have different kernels. 

We haven't yet completed this final step using the methods of this thesis, but there are various ways to conclude the desired result, for example Quillen's original proof or using the Even's norm.

\end{remark}

Duflot also showed that the associated primes of $H_G^*$ are given by restricting to $p$-tori, but there is no known group theoretic description of what $p$-tori give associated primes. Our proof of \ref{HSMtoral} is essentially a reformulation of Duflot's original proof.

\begin{thm}[Duflot \cite{duflotassociatedprimes} Theorem 3.1]
The associated primes of $H_G^*$ all come from restricting to  $p$-tori of $G$.
\end{thm}
\begin{proof}
This follows from \ref{faithfullyflat} and \ref{HSMtoral}.
\end{proof}

In the following result, Okuyama in \cite{okuyama2010remark} Theorem 0.1 shows the equivalence of $1$ and $3$ for finite groups, and the equivalence of $2$ and $3$ is a special case of a result due to Kuhn in \cites{kuhnprimitives,kuhnnilpotence} Theorem 2.13, Propostion 2.8. Kuhn's results include the compact Lie group case.
\begin{thm}\label{restrictionsonassociatedprimes}
Let $A$ be a $p$-torus of $G$. The following are equivalent.
\begin{enumerate}
\item $A$ represents an associated prime in  $H^*_G$.
\item $A$ represents an associated prime in $H^*_{C_GA}$.
\item We have that $\depth H^*_{C_GA}=\dim A$.
\end{enumerate}
\end{thm}
\begin{proof}
This follows from \ref{assoicatedprimesforHSM}, \ref{faithfullyflat}, and our description of the fixed points of the $A$-action on $U(V)/S$.
\end{proof}

\begin{remark}
Note that this gives some restrictions on which $p$-tori can represent associated primes: such a $p$-torus must be the maximal central $p$-torus in its centralizer. This for example rules out all the $p$-tori of $\ZZ/p \wr \ZZ/p$ of rank from $3$ to $p-1$.
\end{remark}

\begin{thm}[Symonds \cite{symondsregularity} Theorem 7.1]\label{regularityforgroupcohomology}
For $G$ a compact Lie group with orientable adjoint representation, $\reg H^*_G= - \dim G$.
\end{thm}
\begin{proof}
Since $H_S^* G \backslash U(V)\cong H_G^*U(V)/S$, by \ref{regularityforHSM}, $\reg H_S^* G \backslash U(V)=\reg H_G^*U(V)/S$ and this quantity is less than or equal to $\dim U(V) -\dim G$. 
Hwoever as $H_G^*$-modules, $H_G^*U(V)/S \cong H_G^* \otimes H^* U(V)/S$. Now, because   $\reg H^*U(V)/S= \dim U(V)$ and regularity is additive under tensor products, we have that $\reg H_G^* \le -\dim G$.
{}
That $- \dim G \le \reg H^*_G$ follows from the local cohomology spectral sequence of \cite{greenleesssforcompactliegroups} Theorem 1.1; it is this direction that requires the adjoint representation to be orientable. 
\end{proof}

\begin{thm}[Symonds \cite{symondsregularity}]\label{strongregularityforgroupcohomology}
For $G$ a compact Lie group, if  $a_i(H^*_G)=-i-\dim G$ for some $i$, then $i$ must be the rank of a maximal $p$-torus in $G$.
\end{thm}
\begin{proof}
This follows from \ref{strongregularityforHSM} and the fact that the minimal elements of $\Fix(G \backslash U(V))$ are in bijection with the maximal $p$-tori of $G$.
\end{proof}
\begin{remark}
Symonds uses this without explicitly stating it in his proof of a special case of the Strong Regularity conjecture.
\end{remark}

\begin{thm}\label{detectiononcentralizers}
If $\depth H^*_G=d$, then $H^*_G \to \prod_{\rank E=i} H^*_{C_G E}$ is injective.
\end{thm}
\begin{proof}
By \ref{detectionforHSM} and our analysis of the elements of $\Fix(G \backslash U(V))$ and $\Fix(U(V)/S$, we have an injective map $H_G^*U(V)/S \to \prod_{\rank E=i} H^*_{C_GE} C_{U(V)} E/^uS$. This fits into a diagram:

\begin{tikzcd}
H_G^*U(V)/S \rar & \prod_{\rank E=i} H^*_{C_GE} C_{U(V)} E/^uS \\
H_G^* \rar \uar &  \prod_{\rank E=i} H^*_{C_GE} \uar
\end{tikzcd}.

Both vertical arrows are injections by \ref{faithfullyflat}, so because the top arrow is an injection the bottom arrows is too.

\end{proof}
\begin{remark}
Carlson in \cite{carlsondepthconjecture} Theorem 2.3 proves this theorem when $G$ is a finite group. The unstable algebra techniques of Henn, Lannes, and Schwartz \cite{hlsunstablemodulesandmodpcohomology} could also prove this theorem, presumably including the case of compact Lie groups, see \cite{poulsenmastersthessis} Theorem 3.47, which gives a proof of Carlson's detection theorem for finite groups using unstable algebra technology.
\end{remark}

Theorems \ref{detectiononcentralizers} and \ref{restrictionsonassociatedprimes} can be combined to show that $H_G^*$ is detected on the family of centralizers of $p$-tori representing associated primes, and this family is minimal.

First recall the definition of a family of subgroups:
\begin{Def}
For $G$ a compact Lie group a collection  $\mathcal{F}$ of subgroups of $G$ is called a family if $\mathcal{F}$ is closed under  subconjugacy, in other words if $H \in \mathcal{F}$ and $K^g<H$ for some $g \in G$ implies that $K \in \mathcal{F}$ as well.
\end{Def}

\begin{Def}
Let $\mathcal{C}_{\ass}$ denote the family generated by the set of subgroups $C_GE$, where $E$ represents an associated prime in $H_G^*$.
\end{Def}

\begin{cor}
We have that $H_G^*$ is detected on $\mathcal{C}_{\ass}$, and if $\mathcal{F}$ is any other family generated by centralizers of $p$-tori, then $\mathcal{C}_{\ass} \subset \mathcal{F}$.
\end{cor}
\begin{proof}
First, we show that $G$ is detected on  $\mathcal{C}_{\ass}$. Suppose that $\depth H_G^*=d$, so $H^*_G$ is detected on the centralizers of rank $d$ $p$-tori. If $E$ is a rank $d$ $p$-torus, if $E$ doesn't represent an associated prime in $H_G^*$, then $\depth H^*_{C_G E}  >d$, so $H^*_{C_{G}E}$ is itself detected on centralizers of higher rank $p$-tori. So, because $C_{C_G E} A= C_GA$ for a $p$-torus $A$ in $C_GE$ in any detecting family of centralizers of $p$-tori we can always replace $C_GE$ by a collection of centralizers of higher rank $p$-tori, unless $E$ represents an associated prime.

To show that this family is the minimal family of centralizers of $p$-tori that detects cohomology, suppose that $\mathcal{F}$ is another such family.

Then $\ass_{H_G^*} H_G^* \subset \ass_{H_G^*} \prod_{H \in \mathcal{F} } H^*H= \cup_{H \in \mathcal{F} } \ass_{H_G^*} H^*_H$, and $\ass_{H^*_G}H^*H$ is mapped onto by $\ass_{H^*_H} H^*_H$. So, if $E$ represents an associated prime in $H^*_G$, then this associated prime must be in some $\ass_{H^*_G}H^*_H$, so there is some $H \in \mathcal{F}$ containing $E$. But by assumption $H=C_G A$ for some $p$-torus $A$, so $C_G E \subset H$, which shows that $\mathcal{C}_{\ass} \subset \mathcal{F}$.
\end{proof}
\begin{remark}
Note that $\mathcal{C}_{\ass}$ is not the minimal detecting family, but merely the minimal detecting family generated by centralizers of $p$-tori. Indeed, if the Duflot bound for depth is sharp (and $G$ is a $p$-group), then as a central $p$-torus will represent an associated prime in $H_G^*$, $\mathcal{C}_{\ass}$ is the family of all subgroups. The family of all subgroups is the minimal detecting family if and only if there is  essential cohomology, i.e.  cohomology classes that restrict to zero on all subgroups. 

However, there are examples of groups where the Duflot bound is sharp and that do not have essential cohomology. The semidihedral group of order $16$ is one such example, it is detected on the family of all proper subgroups (i.e. it has no essential cohomology), but the depth is one and $\mathcal{C}_{\ass}$ is the family of all subgroups. If there is essential cohomology, then by \ref{detectiononcentralizers} the Duflot bound must be sharp and therefore $\mathcal{C}_{\ass}$ is the minimal detecting family. 

It would be interesting to have an understanding of those groups that have no essential cohomology and where the Duflot bound for depth is sharp. In the special case where $H_G^*$ is Cohen-Macaulay, Adem and Karagueuzian in \cite{ademessentialcohomology} Theorem 2.1
give a group theoretic criterion for when there is essential cohomology and prove that under the assumption that $H_G^*$ is Cohen-Macaulay having essential cohomology is equivalent to the Duflot bound being sharp.
\end{remark}

We now give a proof of the result due to Kuhn that Carlson's conjecture is true in a special case.
\begin{thm}\label{carlsonsconjecturewhenduflotboundissharp}
If the Duflot bound is sharp for the depth of $H^*_G$, then the maximal central $p$-torus of $G$ represents an associated prime in $H_G^*$. 

In the case when $G$ is finite and the strengthened Duflot bound is sharp (see \ref{strengtheningofduflotbound}), then the maximal central $p$-torus of a $p$-Sylow of $G$ represents an associated prime.
\end{thm}
\begin{proof}
The first sentence of the theorem follows immediately from \ref{restrictionsonassociatedprimes}, where we let $A$ be the maximal central $p$-torus.

For the case when the strengthened Duflot bound is sharp, let $A$ be the maximal central $p$-torus of a $p$-Sylow, let $\dim A=d$, and suppose that the depth of $H_G^*$ is $d$. Then $C_GA$ has index prime to $p$ in $G$, so $H_G^*$ is a direct summand of $H_{C_GA}^*$. Therefore, if $\HH^{d}H_G^*$ is nonzero, so is $\HH^d H^*_{C_GA}$. Therefore the Duflot bound is sharp for $H_{C_GA}^*$, so $A$ represents an associated prime.
\end{proof}
\begin{remark}
For finite groups, Green proved this for $p$-groups in \cite{greencarlsonsconjecture} Theorem 0.1 and Kuhn extended the result to all finite groups in \cite{kuhnprimitives} Theorem 2.13, and to compact Lie groups in \cite{kuhnnilpotence} Theorem 2.30 .
\end{remark}

\subsection{i-trivial bundles}\label{itrivialbundles}
Here we explore the relationship on cohomology and local cohomology when we have an $S$-equivariant principal $K$-bundle $M \to N$, where $M$, $N$ are smooth $S$-connected manifolds. Our goal is to describe a manifestation of the structures studied in Section \ref{KDuflot} in equivariant topolgy.

Recall that $S$ denotes a $p$-torus.

\begin{Def}
For $K$ a finite group an $S$-equivariant principal $K$-bundle $\pi:M \to N$ is \emph{$i$-trivial} if for all $Y \in \Fix(N)_{\ge i}$ $\pi^{-1}Y \to Y$ is a trivial $K$-bundle.
\end{Def}
% \begin{remark}
% If $M$ is a left $S$-space and $K$ any finite group, then the trivial $K$-bundle $M \times K \to M$ is $i$-trivial for any $i$. Of course the interesting cases are when the bundle is not trivial. We will construct examples of $i$-trivial bundles by group extensions that we also call $i$-trivial momentarily.
% \end{remark}

Let $K \to M \to N$ be an $i$-trivial $S$-equivariant bundle.

\begin{lemma}\label{property}
If an equivariant submanifold $W \subset N$ is fixed by $A < S$ with $\rank A \ge i$, so is $\pi^{-1}W$, and   if  $Y$ is a component of $M^A$, with $\rank A \ge i$, then $\pi$ restricted to $Y$ is a diffeomorphism onto a component of $N^A$.
\end{lemma}
\begin{proof}
For the first part, if $x\in N$ is fixed by $A$, then $x \in W$ where $W$ is a component of $N^A$. Then $\pi^{-1} W \to W$ is an $S$-equivariant trivial principal $K$-bundle, so the fiber over $x$ is also fixed by $A$.

For the second part, if $Y$ is a component of $M^A$, then $\pi(Y)$ is connected and fixed by $A$, so $\pi(Y) \subset W$, for some $W \in \Fix(N)_{\ge i}$. Then there is a pullback square:
\begin{tikzcd}
\pi^{-1}\pi(Y) \dar \rar & \pi^{-1} W \dar \\
\pi(Y) \rar & W
\end{tikzcd} which is isomorphic to \begin{tikzcd}
\pi(Y) \times K \dar \rar & W \times K \dar \\
\pi(Y) \rar & W
\end{tikzcd}

We see that $Y \subset \pi^{-1} \pi(Y)$, and as $Y$ is connected $Y$ is contained in a unique component $W \times k \subset W \times K $. But then as $Y$ is a component of $M^A$ and $Y$ is fixed by $A$, we must have that that $Y$ is all of $W \times k$
\end{proof}

\begin{cor}
An $i$-trivial map $\pi:M \to N$ induces a map $\pi_*: \Fix(M)_{\ge i} \to \Fix(N)_{\ge i}$ by $Y \mapsto \pi(Y)$, and this map gives $\Fix(M)_{\ge i} \to \Fix(N)_{\ge i}$ the structure of a principal $K$-bundle, as defined in section \ref{KDuflot}.
\end{cor}
\begin{proof}
The previous lemma shows that $\pi_*$ defines a map $\Fix(M)_{\ge i} \to \Fix(N)_{\ge i}$ compatible with the ranking. To see that it is a map of posets we must show that if $Y \subset Y'$ then $\pi(Y) \subset \pi(Y')$, which is completely obvious. 

The previous lemma also shows that over each $W \in \Fix(N)_{\ge i}$, $\pi^{-1}W \cong  W \times K$, but we also need to see that this is true over each chain in $\Fix(N)_{\ge i}$. So, given a chain  $W_1 \to \dots \to W_k$ is a chain in $\Fix(N)_{\ge i}$, consider $\pi^{-1}W_1$. Choosing an isomorphism $\pi^{-1}W_1 \cong W_1 \times K$ determines an isomorphism $\pi^{-1}W_i \times K$ for all $i$ because $W_1 \times e$ lies in a unique component of $\pi^{-1}W_i$, so we have shown that $\pi_*: \Fix(M)_{\ge i} \to \Fix(N)_{\ge i}$ is a covering map.

To see that it is a principal $K$-bundle, we note that the $K$-action on $M$ induces a $K$ action on $\Fix(M)_{\ge i}$, and that $\Fix(M)_{\ge i}/K=\Fix(N)_{\ge i}$.
\end{proof}

Denote the functors filtering  $H_S^*M$ and  $H_S^*N$ by $F$ and $G$ respectively.

\begin{lemma}
The $K$-action on $M$ makes $(F_i H_S^*M, \Fix(M)_{\ge i}, F )$ a $K$-Duflot module.
\end{lemma}
\begin{proof}
We only need to define the natural transformations $Fk \mapsto F$, and these comes from the restriction maps $H_S^*Y \from H_S^*kY$ induced by the $K$ action, that they satisfy the axioms from \ref{thedefinitionofanequivariantmorphism} is immediate from the fact that the maps are coming from a group action.
\end{proof}

\begin{thm}
There is an equivariant morphism $$\pi: (F_i H_S^*M, \Fix(M)_{\ge i}, F) \to (F_i H_S^*N, \Fix(N)_{\ge i}, G)$$ making $$\pi: (F_i H_S^*M, \Fix(M)_{\ge i}, F) \to (F_i H_S^*N, \Fix(N)_{\ge i}, G)$$ into a $K$-bundle, and the induced map $F_i H_S^*N \to F_i H_S^*M$ is the map coming from restriction $H_S^*N \to H_S^*M$.
\end{thm}
\begin{proof}
The only piece of data we haven't defined yet is the natural transformation $G\pi_* \Rightarrow F$, and this also comes from the restriction maps $H_S^*(\pi(Y) ) \to H_S^*Y$. That all the given data satisfies the requirements of \ref{thedefinitionofanequivariantmorphism} is immediate.

To see that the map $F_i H_S^*N \to F_i H_S^*M$ induced by the equivariant morphism agrees with the map induced by the map of spaces $M \to N$, we appeal to the uniqueness result of \ref{lemmaboutamorphisminducingmorphims}
\end{proof}

\begin{thm}\label{spectralsequenceforitrivialbundle}
There is a spectral sequence  with $E_2 =H^p(K, \HH^q(F_i H_S^*N))$ converging to $ \HH^{p+q}F_i H_S^*M$.
\end{thm}
\begin{proof}
This follows immediately from \ref{spectralsequenceforKduflotmodules}.
\end{proof}

Now, we can connect this to the group theory.
\begin{Def}
For $i$  less than  or equal to  the $p$-rank of $G$, we say that an extension $1 \to H \to G \to K \to 1$ with $K$ finite is \emph{$i$-trivial} if for all $j \ge i$, if $E$ is a rank $j$ $p$-torus of $G$, then $C_GE < H$.
\end{Def}
\begin{remark}
The main example we will be concerned with are iterated wreath products. For example, the extension $1 \to (\ZZ/p)^p \to \ZZ/p \wr \ZZ/p \to \ZZ/p \to 1$ for $p \ge 3$ is $3$-trivial. 
\end{remark}

\begin{prop}\label{propthatitrivialthingsexist}
If $H \to G \to K$ is $i$-trivial, and $V$ is a faithful representation of $G$, then $K \to H\backslash U(V) \to G \backslash U(V)$ is $i$-trivial.
\end{prop}
\begin{proof}
To show this, we use our description of the fixed points of $G \backslash U(V) ^A$, where $A <S$. Recall that the connected component of $Gu \in G \backslash U(V)^A$ is isomorphic to $C_{G}^uA \backslash C_U(V) ^u A$. By the $i$-triviality assumption, this is isomorphic to $C_{H}^uA \backslash C_{U(V)} ^u A$. 

Now the points of $H \backslash U(V)$ lying over $Gu$ are $\{Hku: k \in K \}$. Under the $i$-triviality assumption each $Hku$ is also fixed by $A$, and the connected component of each $Hku$ is $C_H ^{ku}A \backslash C_{U(V)}^{ku}$. Recall how $S$ acts on one of the $C_{H}^uA \backslash C_{U(V)} ^u A$: it is via the natural action of $S^u$ on $C_{H}^uA \backslash C_{U(V)} ^u A$, and the twist map $S \to S^u$. Therefore these are all isomorphic as $S$-manifolds. 

In order to complete the proof, we just need to show that these components are all disjoint. So, suppose that $Hgu=Hg'uc$, where $c \in C_{U(V)}A$. We wish to show that $g$ and $g'$ have the same image in $K$. We have that $Hg=Hg' ucu^{-1}$. So, $ucu^{-1} \in G$. We also have that $ucu^{-1} \in C_{U(V)}^{^u A}$, so $ucu^{-1} \in C_G ^u A$. Therefore by $i$-triviality, $ucu^{-1} \in H$ as well, so $Hg=Hg'h$, so $g$ and $g'$ have the same image in $H$ and we are done.
\end{proof}

Putting it all together, \ref{spectralsequenceforitrivialbundle} gives us a spectral sequence computing the local cohomology of $H_S^* G \backslash U(V)$ when we have an $i$-trivial extension.

\begin{thm}\label{spectralsequenceforitrivialextension}
If $1 \to H \to G \to K \to 1$ is $i$-trivial, and $V$ a faithful representation of $G$, then  there is a spectral sequence starting at the $E_2$ page: $H^p(K, \HH^q F_i H_S^* H \backslash U(V)) \Rightarrow \HH^{p+q} F_i H_S^* G \backslash U(V)$.
\end{thm}

\subsection{The top $p-2$ local cohomology modules of $H_S^*W(n) \backslash U(V)$}\label{wreathproducts}

Here, we give an application of the theory we have developed so far to do some computations in local cohomology.
Our goal will be to exploit $i$-triviality to get at the top local cohomology modules of $H_S^*W(n) \backslash U(V)$, where $W(1)$ is $\ZZ/p$ and for $n>1$ $W(n)$ is $W(n-1) \wr \ZZ/p$, and $V$ is any faithful representation of $W(n)$.

Recall that $W(n)$ is the $p$-Sylow of $S_{p^{n}}$. The depth of $H^*_{W(n)}$ is known by \cite{carlsonandhennwreathproducts} Theorem 2.1, it is $n$. The Krull dimension is $p^{n-1}$, so these groups have a large difference between their depth and dimension and consequently a lot of room for nontrivial local cohomology. However, other than the regularity theorem, which was proved for these groups by Benson in \cite{bensonontheregularityconjecture}, nothing is known about the structure of their local cohomology modules.

For a  $p$-group $H$, let $r(H)$ (the rank of $H$) be the maximal rank of a $p$-torus of $H$, i.e. the dimension of $H^*_H$. Let $G= H \wr \ZZ/p$ be the split extension $H^p \to G \to \ZZ/p$, and denote the generator of $\ZZ/p$ by $\sigma$.

\begin{thm}\label{trivialityforwreathproducts}
We have that $H^p \to G \to \ZZ/p$ is $r(H)+2$ trivial.
\end{thm}
We'll prove this in a series of lemmas.
\begin{lemma}
Any subgroup of $H \wr \ZZ/p$ not contained in $H^p$ is generated by a subgroup $K$ of $H^p$ and an element of $H \wr \ZZ/p - H^p$.
\end{lemma}
\begin{proof}
Suppose that $g'$, $h'$ are elements of $H \wr \ZZ/p - H^p$. We will show that the subgroup generated by $g'$ and $h'$ is equal to the subgroup generated by $g'$ and $k$, where $k \in H^p$.

First, write $g'=g'' \sigma^j$, and $h'=h'' \sigma^l$, where $g'',h'' \in H^p$ and $j,l \not\equiv 0 \pmod{p}$. Then there are $m,n \not\equiv 0 \pmod {p}$ so that $(g')^m=g \sigma$, $(h')^n=h\sigma$. 
But since $H$ and therefore $H \wr \ZZ/p$ are $p$-groups, $(g')^m$ and $(h')^n$ generate the same subgroup as $g'$ and $h'$, so we can reduce to the case where we have two elements of the form $g\sigma$ and $h\sigma$. 
But then $h\sigma(g \sigma)^{-1} \in H^p$, and we are done.
\end{proof}

\begin{lemma}
Given an element $g \sigma^j$ of $H \wr \ZZ/p - H^p$, the maximal rank of a $p$-torus of $H^p$ centralized by $g \sigma^j$ is $r(H)$.
\end{lemma}
\begin{proof}
We want to classify the elements of $H^p$ that commute with $g \sigma^j$. As above, by raising $g \sigma^j$ to the appropriate power we can  reduce to the case that $j=1$. Then if $h=(h_1,\dots,h_p) $ commutes with $g\sigma=(g_1,\dots,g_p)\sigma$, we have that 
$h g \sigma h^{-1}=g\sigma$, so $h g (\sigma \cdot h^{-1} ) \sigma=g \sigma$, so $h g (\sigma \cdot h)^{-1}=g$. 

Therefore we have the equations: 
\begin{equation}
\begin{split}
h_1 g_1 h_p^{-1}&=g_1 \\
h_2 g_2 h_1^{-1}&=g_2 \\
\dots \\
h_i g_i h_{i-1}^{-1}&=g_i \\
\dots \\
h_p g_p h_{p-1}^{-1}&=g_p
\end{split}
\end{equation}

Therefore, $h_1$ determines at most one $h$ that commutes with $g \sigma$. But if $g \sigma$ is to centralize a $p$-torus of $H^p$, all the choices for the first coordinate of $h$ must commute with one another, so they must lie in a $p$-torus of $H$, giving us our result.
\end{proof}

\begin{proof}[Proof of the theorem]
We must show that if $E$ is a $p$-torus of rank greater than $r(H)+1$ then the centralizer of $E$ lies entirely in $H^p$.

First, we will show that $E$ lies in $H^p$. If not, by the first lemma $E$ is generated by $E' < H^p$ and an element $g$ of $H \wr \ZZ/p - H^p$. But $E'$ must then be a $p$-torus of rank greater than $r(H)$ that commutes with an element of $H \wr \ZZ/p - H^p$, which contradicts the second lemma.

Now the second lemma finishes the proof.
\end{proof}

\begin{lemma}
If $H \to G$ is $i$-trivial, then $H^n \to G^n$ is $(n-1)r(H)+i$ trivial.
\end{lemma}
Recall that if $H$ is $i$-trivial in $G$ by definition $i$ is less than or equal to the $p$-rank of $G$, so the $p$-rank of $H$ is equal to the $p$-rank of $G$.
\begin{proof}
The $i$-triviality condition ensures that for $E$ a $p$-torus of rank greater than or equal to $(n-1)r(H)+i$ in $G^n$, the rank of $\im \pi_j E$ is greater than or equal to $i$ for each $j$ ($\pi_j$ is the $j^{th}$ projection map). So, if $c \in G^n$ commutes with $E$, then as $\pi_j(c)$ commutes with $\pi_jE$, so $\pi_j(c) \in H$, and $c \in H^n$.
\end{proof}

\begin{lemma}\label{normalsubgroupsofwreathproducts}
If $H$ is normal in $G$, then $H^p$ is normal in $G \wr \ZZ/p$.
\end{lemma}
\begin{proof}
This follows because $H^p$ is invariant under the $\ZZ/p$ action on $G^p$ and is normal in $G^p$.
\end{proof}

\begin{cor}\label{trivialityforsubgroupsofwreathproducts}
If $H$ is $i$-trivial in $G$, and $(p-1)r(H)+i$ is greater than $r(G)+2$, then $H^p$ is $(p-1)r(H)+i$-trivial in $G \wr \ZZ/p$
\end{cor}

%\begin{cor}
%We have that $(\ZZ/p \wr \ZZ/p)^p$ is $p+2$-trivial in $(\ZZ/\p \wr \ZZ/p ) \wr \ZZ/p$, and $(\ZZ/p)^{p^2}$ is $p^2-(p-3)$-trivial in $(\ZZ/p \wr \ZZ/p) \wr \ZZ/p$.
%\end{cor}
%\begin{proof}
%The theorem gives the first claim. Recall that $\ZZ/p \wr \ZZ/p$ has a unique rank $p$-maximal $p$-torus, this gives us the rank $p^2$ $p$-torus in $(\ZZ/p \wr \ZZ/p) \wr \ZZ/p$. The theorem gives that $(\ZZ/p)^p$ is $3$-trivial in $\ZZ/p \wr \ZZ/p$, so $(\ZZ/p ^p)^p$ is $(p-1)p+3=p^2-(p-3)$ trivial in $(\ZZ/p \wr \ZZ/p) \wr \ZZ/p$ by the above corollary.
%\end{proof}
%
%\begin{prop}
%Let $K$ be $(\ZZ/p \wr \ZZ/p ) \wr \ZZ/p / (\ZZ/p)^{p^2}$. There is a spectral sequence $H^i(K, \HH^q( F_{p^2-(p-3)} H_S^* ( \ZZ/p^{p^2} \backslash U))) \Rightarrow \HH^{i+q} F_{p^2-(p-3)} H^*_S( (\ZZ/p \wr \ZZ/p) \wr \ZZ/p \backslash U)$.
%\end{prop}
%
%Recall that $\HH^j(F_{p^2-(p-3)} H_S^* ( \ZZ/p^{p^2} \backslash U))$ is zero except for possibly $j=p^2-(p-3)$ and $j=p^2$. As long as $p \ge 5$, by examining the highest degree part of $\HH^{p^2} F_{p^2-(p-3)} H_S^* ( \ZZ/p^{p^2} \backslash U))= \HH^{p^2}(H_S^* ( \ZZ/p^{p^2} \backslash U))$ we see that $H^i(K, \HH^{p^2}( F_{p^2-(p-3)} H_S^* ( \ZZ/p^{p^2} \backslash U)))$ and consequently $H^i(K, \HH^{p^2-(p-3)}( F_{p^2-(p-3)} H_S^* ( \ZZ/p^{p^2} \backslash U)))$ are nonzero for infinitely many $i$.
%
%So, $$d_{p-2}: H^i(K, \HH^{p^2}( F_{p^2-(p-3)} H_S^* ( \ZZ/p^{p^2} \backslash U))) \to H^{i+p-2}(K, \HH^{p^2-(p-3)}( F_{p^2-(p-3)} H_S^* ( \ZZ/p^{p^2} \backslash U)))$$ is a surjection for $i=0$ and an isomorphism for $i>0$.

\begin{lemma}
Let $p \ge 5$, and let $W(n)$ denote the $n-1$-fold iterated wreath product of $\ZZ/p$ with itself, with $W(1)= 1 \wr \ZZ/p=\ZZ/p$. Then for $n> 1$, $W(n)$ has a unique, normal, maximal rank $p$-torus $E(n)$ of rank $p^{n-1}$, $E(n)$ is $p^{n-1}-p+3$-trivial in $W(n)$, and $W(n)/E(n) \cong W(n-1)$.
\end{lemma}
\begin{proof}
First, we will define $E(n)$,  show that it is normal, has maximal rank, and that the quotient is $W(n-1)$. Then we will show that it is $p^{n-1}-p+3$ trivial, which implies that it is the unique maximal $p$-torus. 

We define $E(n)$ inductively: $E(1)=W(1)=\ZZ/ p$, and for $n>1$ $E(n)= (E(n-1))^p < W(n-1)^p <W(n)$. This shows that $E(n)$ is normal by \ref{normalsubgroupsofwreathproducts}.  To see that $E(n)$ has maximal rank, we can note that $W(n)$ is the $p$-Sylow of $S_{p^n}$, where it is clear that a maximal rank $p$-torus has rank $p^{n-1}$.  

To determine that the quotient is $W(n-1)$, we again proceed inductively (noting that this is true for $n=2$). Note that we have the map of extensions: 

\begin{tikzcd}
1 \rar & W(n-1)^p \dar  \rar & W(n) \dar \rar & \ZZ/p \dar \rar & 1 \\
1 \rar & W(n-1)^p / E(n) \rar & W(n)/E(n) \rar & \ZZ/p \rar & 1
\end{tikzcd}

So $W(n)/E(n)$ fits into an extension of $\ZZ/p$ by $W(n-1)^p/ E(n-1)^p$, which is $W(n-2)^p$ by hypothesis. Then a section $\ZZ/p \to W(n)$ determines a section of $W(n)/E(n) \to \ZZ/p$, and the induced action of $\ZZ/p$ on $W(n-1)^p/E(n-1)^p$ cyclically permutes the factors, so $W(n)/E(n)$ is $W(n-1)$.

Now, it remains to show that $E(n)$ is $p^{n-1}-p+3$ trivial. First, this is true for $n=2$. We have that $W(2)= \ZZ/p \wr \ZZ/p$ and $E(2)=\ZZ/p^p$, and this is $3$ trivial by \ref{trivialityforwreathproducts}. Now, suppose it is true for $W(n-1)$, and consider $W(n)$. Since $r(E(n-1))=r(W(n-1))$, \ref{trivialityforsubgroupsofwreathproducts} tells us that $E(n)=E(n-1)^p$ is $(p-1)p^{n-2}+p^{n-2}-p+3$ trivial in $W(n)$, and $(p-1)p^{n-2}+p^{n-2}-p+3=p^{n-1}-p+3$.
\end{proof}

Now, applying our result \ref{spectralsequenceforitrivialextension} about spectral sequences for $i$-trivial extensions, we have the following:
\begin{prop}
There is a spectral sequence $$H^i(W(n-1), \HH^j( F_{p^{n-1}-p+3} H_S^*( E(n) \backslash U(V))) \Rightarrow \HH^{i+j}(F_{p^{n-1}-p+3}H^*_{S} (W(n) \backslash U(V))).$$
\end{prop}
Note that $\HH^j( F_{p^{n-1}-p+3} H_S^*( E(n) \backslash U(V))) $ is zero except for $j=p^{n-1}$ and possibly for $j=p^{n-1}-p+3$, that $\HH^{i+j}(F_{p^{n-1}-p+3}H^*_{S} (W(n) \backslash U(V)))$ is zero for $i+j< p^{n-1}-p+3$ and $i+j>p^{n-1}$, and that $\HH^{i+j}(F_{p^{n-1}-p+3}H^*_{S} (W(n) \backslash U(V)))=\HH^{i+j}(H^*_{S} (W(n) \backslash U(V)))$ for $i+j>p^{n-1}-p+3$. The last equality follows from \ref{whytruncationsaregood}.

This tells us that the spectral sequence is concentrated in two rows, so the only differential is a $d_{p-3+1}: E_{p-3+1}^{i,p^{n-1}} \to E_{p-3+1}^{i+p-3+1,p^{n-1}-(p-3)}$. Therefore this differential must be an isomorphism for $i>0$ and a surjection for $i=0$, which tells us that: 
\begin{multline}$$\HH^{i} H_S^* (W(n) \backslash U(V))= \\ H^{i-(p^{n-1}-(p-3))}(W(n-1), \HH^{p^{n-1}-(p-3)}( F_{p^{n-1}-(p-3)} H_S^*( E(n) \backslash U(V))))$$\end{multline} for $p^{n-1}-(p-3) < i < p^{n-1}$, and: 
\begin{multline}$$
\HH^{p^{n-1}} H_S^*( W(n) \backslash U(V)) = \\ 
H^{(p-3)}(W(n-1), \HH^{p^{n-1}-(p-3)}( F_{p^{n-1}-(p-3)} H_S^*( E(n) \backslash U(V)))) \\ \oplus (\ker: d_{p-2}:E_2^{0,p^{n-1}} \to E_2^{p-2,p^{n-1}-(p-3)}). $$\end{multline}

We would like to have some expression of $\HH^{i}( H_S^* W(n) \backslash U(V))$ that doesn't reference the Duflot filtration, which we can achieve if we can relate \\  $H^{i-(p^{n-1}-(p-3))}(W(n-1), \HH^{p^{n-1}-(p-3)}( F_{p^{n-1}-(p-3)} H_S^*( E(n) \backslash U(V))))$ to the group cohomology of $W(n-1)$ with coefficients in $\HH^{p^{n-1}} H_S^* (E(n) \backslash U(V))$. Fortunately, we are able to do this.

Recall how the spectral sequence at hand is constructed: it is one of the hypercohomology spectral sequences for $W(n-1)$ acting on the Duflot complex for one level of $H_S^*( E(n) \backslash U(V))$. We know what it is converging to since the Duflot complex for this level of $H_S^* E(n) \backslash U(V)$ is a complex of free $W(n-1)$ modules, and the fixed point set is the Duflot complex for $H_S^*W(n) \backslash U(V)$. 

If instead of taking the hypercohomology spectral sequence we constructed a hyper-Tate cohomology spectral sequence, we would have: $$E_2^{i,j}= \widehat{H}^i( W(n), \HH^j ( F_{p^{n-1}-(p-3)} (H_S^*E(n) \backslash U(V)) )).$$ This agrees with the $E_2$ page of the previous spectral sequence for $i>0$, and it converges to $0$ because the Tate cohomology of a free module is zero. This tells us that for all $i$, $\widehat{H}^i( W(n-1), \HH^{p^{n-1}}( H_S^* E(n) \backslash U(V))) \to[\sim] \widehat{H}^{i+p-2} (W(n-1), \HH^{p^{n-1}-(p-3)}( F_{p^{n-1}-(p-3)} H_S^* E(n) \backslash U(V)))$.

This gives us the following computation for the local cohomology of $H_S^* W(n) \backslash U(V)$. 

For a $G$-module $A$, let $N: A \to A^G$ denote the map $a \mapsto \sum_{g \in G} g \cdot a$,  and recall that the kernel of the map $H^0(G,A) \to \widehat{H}^0(G,A)$ is $\im N$.
\begin{thm}
For $ 0< i < p-3$, $$\HH^{p^{n-1}-(p-3)+i} H_S^*( W(n) \backslash U(V))\cong \widehat{H}^{i-(p-2)}(W(n-1), \HH^{p^{n-1}} H_S^* E(n) \backslash U(V)). $$ This isomorphism is as $H_S^* W(n) \backslash U(V)$-modules.

For the top local cohomology, we have that $\HH^{p^{n-1}}( H_S^* W(n) \backslash U(V))$ is isomorphic as vector spaces to $\widehat{H}^{-1}(W(n-1), \HH^{p^{n-1}}(H_S^* E(n) \backslash U(V))) \oplus N( \HH^{p^{n-1}} (H_S^* E(n) \backslash U(V))).$
\end{thm}
\begin{proof}
Everything follows from our computation with Tate cohomology, the only part that isn't immediate is the second half of the direct sum in the top local cohomology, and this comes from our identifying the kernel of the map $d_{p-2}$ on $E_2^{0,p^{n-1} }$ of the original spectral sequence with the kernel from $H^0$ to $\widehat{H}^0$.
\end{proof}

By  \cite{dwyergreenleesiyengar} Propostion 9.4 there is a local cohomology spectral sequence $\HH^*(H_S^* E(n) \backslash U(V)) \Rightarrow (H_S^*E(n) \backslash U(V))^*$. But $H_S^* E(n) \backslash U(V)$ is Cohen-Macaulay, so the spectral sequence collapses and we have that: $$\HH^{p^{n-1}} H_S^* E(n)\backslash U(V) = \Sigma^{-p^{n-1}+d} (H_S^* E(n) \backslash U(V))^*$$ Here $d=\dim U(V)$.

So, we can rewrite the computation of the top local cohomology modules of $H_S^* W(n) \backslash U(V)$ without reference to local cohomology.

\begin{thm}\label{computationoflocalcohomologyofH_SU(V)/W(n)}
Denote the dimension of $U(V)$ by $d$. 

For $ 0< i < p-3$, we have an isomorphism $H^*_{W(n)}$-modules: $$\HH^{p^{n-1}-(p-3)+i} H_S^*( W(n) \backslash U(V))=\widehat{H}^{i-(p-2)}(W(n-1), \Sigma^{-p^{n-1}+d} (H_S^* E(n) \backslash U(V))^*). $$ 

For the top local cohomology, we have that $\HH^{p^{n-1}}( H_S^* W(n) \backslash U(V)) \cong \widehat{H}^{-1}(W(n-1), \Sigma^{-p^{n-1}+d} (H_S^* E(n) \backslash U(V))^*)  \oplus N( \Sigma^{-p^{n-1}+d} (H_S^* E(n) \backslash U(V))^*).$ This isomorphism is only as graded $\FF_p$ vector spaces, there is an extension problem to solve to compute the $H_{W(n)}$-module structure.
\end{thm}

As mentioned in the introduction, we can now show that there are groups  whose cohomology has arbitrarily long sequences of nonzero local cohomology modules.
\begin{cor}\label{abandofnonvanishinglocalcohomology}
For $0 \le i <p-3$, $\HH^{p^{n-1}-i}(H^*_{W(n)}) \not= 0$.
\end{cor}
\begin{proof}
That the top local cohomology is nonzero is immediate from the fact that the Krull dimension is equal to the top degree in which local cohomology is nonvanishing, so we just need to show the result for $0<i<p-3$.

By \ref{faithfullyflat} is enough to show that the result is true with $H_S^*( W(n) \backslash U(V))$ in place of $H^*_{W(n)}$.

For this, we can use \ref{computationoflocalcohomologyofH_SU(V)/W(n)}. This tells us that: $$\HH^{p^{n-1}-(p-3)+i}(H_S^* W(n) \backslash U(V))=\widehat{H}^{i-(p-2)}(W(n-1); \Sigma^{-p^{n-1}+d} (H_S^*E(n) \backslash U(V))^*).$$ 

Recall that local cohomology is bigraded. We have that: $$\HH^{ p^{n-1}-(p-3)+i,-p^{n-1}+d}(H_S^* W(n) \backslash U(V))=\widehat{H}^{i-(p-2)}(W(n-1); \FF_p).$$ But $H^*_{W(n-1)}$ is nonzero in all degrees since it is a $p$-group, so we are done.
\end{proof}
We have proved:
\begin{prop}\label{mysteriousgrap}For each $p \ge 5$, for each $n$ there exists a $p$-group $G$ and an $i$ so that $\HH^{i+j}(H_G^*)\not=0$ for all $0<j<n,$  and so that $i+j$ is not the dimension of an associated prime.
\end{prop}

The regularity theorem states that for each finite group $G$, $\HH^{i,j}H^*_G=0$ for $j>-i$ and  for some $i$ $\HH^{i,-i} H_G^* \not=0$. In fact, it is known that if $r$ is the dimension $H_G^*$, then $\HH^{r,-r}H_G^* \not=0$, and it is conjectured that for all other $i$ we have that $\HH^{i,-i} H_G^*=0$. 

In fact, for the groups $W(n)$ in the range we have been studying more is true.

\begin{cor}\label{themysterioustriangleofzeroes}
For $0 \le  i < p-3$, $\HH^{p^{n-1}-i,j}H^*_{W(n)}=0$ for $j>-p^{n-1}$. 
\end{cor}
\begin{proof}
For $i=0$ the result is true by the regularity theorem, so we just need to show the result for $i>0$.

Recall that by \ref{faithfullyflat}, $H_S^* W(n)\backslash U(V) \cong H^*_{W(n)} \otimes H^* W(n) \backslash U(V)$ as $H^*_{W(n)}$ modules. So, $\HH^*( H_S^*(W(n) \backslash U(V)) \cong \HH^*(H^*_{W(n)}) \otimes H^* W(n)\backslash U(V)$. So, as $W(n) \backslash U(V)$ is an oriented $N$ dimensional manifold (where $N=\dim U(V)$) the top nonvanishing degree of its cohomology is $N$. Therefore the top nonzero degree of $\HH^i H_S^* W(n) \backslash U(V)$ is $N$ plus the top nonzero degree of $\HH^i H^*_{W(n)}$. 

So, we need to show that for $0<i<p-3$, $\HH^{p^{n-1}-i,j+N}H^*_{S} W(n) \backslash U(V)=0$ for $j>-p^{n-1}$. 
For this, by \ref{computationoflocalcohomologyofH_SU(V)/W(n)} we have that  $\HH^{p^{n-1}-(p-3)+i} H_S^*( W(n) \backslash U(V))=\widehat{H}^{i-(p-2)}(W(n-1), \Sigma^{-p^{n-1}+N} (H_S^* E(n) \backslash U(V))^*) $. But since $H_S^* E(n) \backslash U(V)$ is concentrated in positive degrees, its dual is concentrated in negative degrees, so $ \Sigma^{-p^{n-1}+N} (H_S^* E(n) \backslash U(V))^*)$ is zero above degree $-p^{n+1}+N$, and the result follows.
\end{proof}

\bibliography{references}

\end{document}